\RequirePackage{snapshot}
\documentclass[preprint]{elsarticle}
\usepackage{amsmath}
\usepackage{amssymb}
\usepackage{amsthm}
\usepackage{fixme}
\usepackage{standalone}
\usepackage{xspace}
\usepackage[all]{xy}

\fxsetup{status=final,layout=footnote,marginclue}

\author{Chris Le Sueur}
\ead[url]{http://www.maths.bris.ac.uk/~cl7907/}
\address{School of Mathematics, \\University of East Anglia, \\Norwich Research Park, \\Norwich NR4 7TJ, \\United Kingdom}
\title{Determinacy of Refinements to the Difference Hierarchy of Co-analytic Sets}

\def\copyrightnotice{%
	 \begin{minipage}{\textwidth}
		 \footnotesize\itshape
		 Accepted August 2017.  This manuscript is made available under the Creative Commons CC-BY-NC-ND 4.0 license: \url{http://creativecommons.org/licenses/by-nc-nd/4.0/}%
	 \end{minipage}
}
\makeatletter
\def\ps@pprintTitle{%
 \let\@oddhead\@empty
 \let\@evenhead\@empty
 \def\@oddfoot{\copyrightnotice}
 \let\@evenfoot\@oddfoot}
\makeatother

\newcommand{\ZFC}{\ensuremath{\mathsf{ZFC}}\xspace}

\newcommand{\KP}{\ensuremath{\mathsf{KP}}\xspace}
\newcommand{\Det}{\ensuremath{\operatorname{Det}}}
\newcommand{\hyp}{\ensuremath{\operatorname{-}}}
\newcommand{\ot}{\ensuremath{\operatorname{ot}}}

\newcommand{\ca}{\ensuremath{{\Pi}^1_1}}
\newcommand{\plays}[1]{\ensuremath{\lceil #1 \rceil}}
\newcommand{\concat}{\ensuremath{\widehat{\:\:\:}}}
\newcommand{\os}{\ensuremath{\omega^2\hyp\Pi^1_1}}
\newcommand{\bs}{\ensuremath{\omega^\omega}}
\renewcommand{\P}{\ensuremath{\mathbb{P}}\xspace}
\newcommand{\wt}{\widetilde}

\newcommand{\skp}{\ensuremath{\KP_1}\xspace}
\newcommand{\nkp}{\ensuremath{\KP_n}\xspace}
\newcommand{\dom}{\ensuremath{\operatorname{dom}}}
\newcommand{\rk}{\operatorname{rk}}
\newcommand{\ran}{\text{ran}}
\newcommand{\texorpdfstring}[2]{#1}
\newcommand{\Trans}{\operatorname{Trans}}
\newcommand{\csym}{\ensuremath{\dot{\vec c}}}
\newcommand{\LF}{\ensuremath{\mathcal L_{F_\lambda}}}

\newtheorem{thm}{Theorem}
\newtheorem*{thm*}{Theorem}
\newtheorem{lem}[thm]{Lemma}
\newtheorem{cor}[thm]{Corollary}
\newtheorem{prop}[thm]{Proposition}

\newtheorem*{fact*}{Fact}

\newtheorem*{claim*}{Claim}

\theoremstyle{definition}

\newtheorem{dfn}[thm]{Definition}
\newtheorem*{dfn*}{Definition}

\numberwithin{thm}{section}

\theoremstyle{remark}

\newtheorem*{rem}{Remark}
\numberwithin{claim}{thm}

\begin{document}
\begin{abstract}
	In this paper we develop a technique for proving determinacy of classes of the form $\omega^2\hyp\Pi^1_1+\Gamma$ (a refinement of the difference hierarchy on $\Pi^1_1$ lying between $\os$ and $(\omega^2+1)\hyp\Pi^1_1$) from weak principles, establishing upper bounds for the determinacy-strength of the classes $\os+\Sigma^0_\alpha$ for all computable $\alpha$ and of $\os+\Delta^1_1$.
	This bridges the gap between previously known hypotheses implying determinacy in this region.
\end{abstract}
\def\sep{\ensuremath{\mathsf{Sep}}}

\renewenvironment{abstract}{\comment}{\endcomment}
\def\sep{\ensuremath{\mathsf{Sep}}}

\maketitle

\section{Introduction}
\label{sec:Introduction}

We work towards proving the following theorem (the relevant definitions can be found in the next section):
\begin{thm}
	\label{thm:main}
	If there exists a non-trivial mouse $\mathcal M$ with measurable cardinal $\kappa$ satisfying the theory $T$, then $\Det(\omega^2\hyp\ca+\Gamma)$ for the following combinations of $T$ and $\Gamma$:
	\begin{enumerate}
		\item $T=\text{``cleverness + there exists a clever mouse,''}$ $\Gamma=\Sigma^0_1$;
		\item $T=\KP+\Sigma_1\hyp\sep$, $\Gamma=\Sigma^0_2$;
		\item $T=\Sigma_2\hyp\KP+\Sigma_2\hyp\sep$, $\Gamma=\Sigma^0_3$;
		\item $T=\Sigma_{n+1}\hyp\KP+\Sigma_{n+1}\hyp\sep$, $\Gamma=n\hyp\Pi^0_3$;
		\item $T=\mathsf{ZFC}^-+\mathcal{P}^\alpha(\kappa)$ exists, $\Gamma=\Sigma^0_{1+\alpha+3}$ for computable $\alpha$;
		\item $T=\ZFC$, $\Gamma=\Delta^1_1$.
	\end{enumerate}
\end{thm}

This is thus an extension of $\omega^2\hyp\Pi^1_1$ determinacy, without requiring all the strength of $0^\dagger$ required to prove determinacy of $(\omega^2+1)\hyp\Pi^1_1$.
The story of $\Det(\omega^2\hyp\Pi^1_1)$ starts with Martin proving in \cite{Martin70} that the existence of a measurable cardinal implies the determinacy of $\mathbf\Pi^1_1$.
The proof uses the measure to ``integrate'' many strategies together in a technique that proved very fruitful and which makes the core of the present paper.
Analysis of the proof allows the result to be broken down to the following:

\begin{thm*}[Martin]~
	\[
		\forall x\in\omega^\omega (x^\sharp\text{ exists} \rightarrow \Det(\Pi^1_1(x))).
	\]
\end{thm*}

Hence $0^\sharp$ implies the determinacy of the lightface co-analytic sets.
Here the role of the measure is seen through the lens of indiscernibility, and this is the form in which our determinacy proof will be.

It was then shown in \cite{Friedman71} that actually $0^\sharp$ implies $\Det(3\hyp\Pi^1_1)$, the third level of the difference hierarchy.
Eventually the full strength of $0^\sharp$ came out in a proof by Martin (available in Martin's unpublished book manuscript \cite{MartinUP} or an account by DuBose \cite{DuBose90}) that

\begin{thm*}[Martin]
	\[
		0^\sharp\text{ exists} \leftrightarrow \Det\left(\textstyle{\bigcup_{\alpha<\omega^2}}\alpha\hyp\Pi^1_1\right).
	\]
\end{thm*}

The much stronger principle of $0^\dagger$ was found to imply $\Det((\omega^2+1)\hyp\Pi^1_1)$, and indeed this is an exact equivalence.
Intermediate results were sought; in \cite{Martin90} Martin proves $\Det(\omega^2\hyp\mathbf\ca)$ from a measurable cardinal and, providing inspiration for our current results, proves $\Det(\Delta(\omega^2+1)\hyp\mathbf\ca)$ and hence $\Det(\omega^2\hyp\mathbf\ca+\mathbf\Delta^1_1)$ from the same hypothesis.
Nonetheless a weaker hypothesis was sought and found in \cite{Welch96}, the paper on which this is based:
\begin{thm*}[Welch]
	There exists a \emph{clever mouse} iff $\Det(\omega^2\hyp\ca)$.
\end{thm*}
A clever mouse is a certain type of iterable model.
Principles such as $0^\sharp$ and $0^\dagger$ can be viewed in terms of mice, as well, so this result in some sense exists on the same continuum.

The difference hierarchy is already a refinement of the projective hierarchy, here, so we consider further refinements; if $\Gamma\subseteq\ca$ is a pointclass, $\omega^2\hyp\ca+\Gamma$ is viewed as $(\omega^2+1)\hyp\ca$ with the final set in the $\omega^2+1$-sequence being constrained to being $\Gamma$.
Using the methods in \cite{Welch96}, we find that such results are closely connected to the determinacy of a class related to $\Gamma$ which we call $\wt\Gamma$.
This is the analogue of $\Gamma$ in a larger space, arising from the auxiliary game used in Section \ref{sec:Determinacy from Indiscernibility} , and in Section \ref{sub:Effective Descriptive Set Theory in Uncountable Spaces} we will spend some time developing the theory of this space.

The way we will prove Theorem \ref{thm:main} is via the following:

\begin{thm}
	\label{thm:ind-to-det}
	Let $\Gamma$ be an arithmetic pointclass.
	Suppose there exists a class $C$ of $\Sigma_n$ generating indiscernibles for a theory $T$ such that for any $\vec c\in[C]^\omega$, any $\wt\Gamma$ game has a $\Sigma_n$-definable winning strategy in the smallest transitive model of $T$ with $\vec c\in T$.
	Then $\Det(\omega^2\hyp\ca+\Gamma)$.
\end{thm}

We will define exactly what generating indiscernibles are in Section \ref{sec:Obtaining Indiscernibles}, where we will also show their existence, and Theorem \ref{thm:ind-to-det} will be proved in Section \ref{sec:Determinacy from Indiscernibility}, where $\wt\Gamma$ will be defined implicitly.
Finally in Section \ref{sec:Individual Determinacy Proofs} we will prove the determinacy of $\wt\Gamma$ for each $\Gamma$ mentioned in \ref{thm:main} in the relevant model.

Section \ref{sub:Dual Classes in the Difference Hierarchy} contains definitions of the difference hierarchy and our refinements to it, as well as some small lemmas that help in establishing the determinacy of dual classes.
Readers familiar with the difference hierarchy may wish to skip that section except for the main Definition \ref{dfn:refined-diff-hierarchy}.
Section \ref{sub:Effective Descriptive Set Theory in Uncountable Spaces} contains essential definitions required to complete the determinacy proofs, but the material closely mirrors the usual effective descriptive set theory.
Sections \ref{sec:Obtaining Indiscernibles} and \ref{sec:Determinacy from Indiscernibility} are independent of one another and can be read in either order, with the exception of Definition \ref{dfn:generating-indiscernibles}.

\section*{Acknowledgements}
The author was supported by a doctoral training grant of the Engineering and Physical Sciences Research Council.
The author would also like to thank his supervisor Philip Welch for many helpful discussions.

\section{Preliminaries}
\label{sec:Definitions and notation}
\label{sec:Preliminaries}

Our set theoretic notation is standard, and we follow \cite{MartinUP} for our determinacy notation.
Thus a game is specified by a \emph{game-tree} $T$, usually $\omega^{<\omega}$, and the payoff set $A\subseteq\plays{T}$.
$\plays T$ is the set of all possible plays, and is equal to $[T]$, the set of branches, if $T$ has no terminal-nodes, which for us will always be the case.
Determinacy for a class of sets $\Gamma$ is denoted $\Det(\Gamma)$.
The Borel hierarchy and co-analytic sets in the space $[T]$ are defined as usual as in \cite{MoschovakisDST}, but since we will be dealing with small models without all of the reals we will need to use the lightface counterparts of these classes.
In Section \ref{sub:Effective Descriptive Set Theory in Uncountable Spaces} this will be generalised to uncountable $T$; once the basic definitions are in order, the results here work the same as in the classical case.

We require familiarity with admissible structures and $L$, where we follow the notation of \cite{Devlin}.
We will also use the Jensen hierarchy $J_\alpha$, but since we mostly deal with admissible structures this can be thought of as $L_\alpha$.
\nkp denotes the axioms of Kripke-Platek set theory augmented by the schemes of $\Sigma_n$ Collection and Separation.
We make extensive use of structures of the form $\langle L_\alpha[A],\in,A\rangle$ for some predicate $A$ (usually a countable sequence of ordinals.)
If we say a structure of this form is admissible, or a model of \nkp, we mean a model in said theory in the language with a predicate for $A$, adjusting the L\'evy hierarchy appropriately.
Most results from the usual $L_\alpha$ hierarchy generalise in the obvious way; in particular we will need $\Sigma_n$ Skolem functions for $n>1$.
Normally this would require the full fine structure theory, but we will always be in a model of at least $\Sigma_{n-1}$ admissibility.
Thus we will freely make use of the fact that ``the $M$-least set $x$ such that $\varphi(x)$'' for $\varphi$ a $\Sigma_n$ formula is $\Sigma_n$ over $M$.

While the statement of Theorem \ref{thm:main} is in terms of mice, we won't need much mouse theory at all, since the mice we use will be models of enough set theory to obviate the need for much of the fine structure theory.
The important part of being a mouse is the property of iterability, i.e. having wellfounded iterated ultrapowers.
Thus when we talk of $M$ being a mouse the important features that we will use are its transitivity and wellfounded ultrapowers, and indeed the hypotheses of the main theorem could just as well be stated in terms of the existence of any transitive, iterable model of enough set theory.
The full theory can be found in \cite{Zeman02}, but we will only need the theory for mice with a single measure, which can be found (albeit with older notation) in \cite{DoddJensen}.

Also touched upon in section \ref{sec:Obtaining Indiscernibles} is the canonical mouse order, $<_*$, the details of which are not important for this paper.

For the first result of the main theorem we refer to the cleverness property defined in \cite{Welch96} that was the key to getting that paper's main result.
Essentially this is a version of Rowbottom's theorem for classes $\Sigma_1$-definable over a model, a principle which we get for free in models of $\Sigma_1$ (or more) separation.
The definition uses the concept of a $Q$-structure, defined at 1.8 in \cite{Welch96}.
Let $M$ be a mouse in the old sense of \cite{DoddJensen}, then:
\begin{dfn}
	Let $M\vDash\text{``$F$ is a normal measure on $\kappa$''}$, where $F$ is, in $V$, the closed-unbounded filter on $\kappa$.
	Then the \emph{$Q$-structure of $M$ at $\kappa$} is:
	\[
		Q^M_\kappa = \langle J^F_\theta,\in,F\cap J^F_\theta\rangle
	\]
	where $\theta$ is largest such that:
	\[
		J^F_\theta\vDash\text{``$F$ is normal, and the club filter on $\kappa$.''}
	\]
\end{dfn}
We will not make extensive use of this material and haven't changed much from \cite{Welch96}, so we omit the majority of the details, commenting rather on the changes where they exist.

The next two subsections introduce more notation as well as the required generalisations.

\subsection{The Difference Hierarchy and Refinements}
\label{sub:Dual Classes in the Difference Hierarchy}

\let\oldbibliography\bibliography
\def\bibliography#1{\relax}
\let\oldsection\section
\let\oldsubsection\subsection
\let\section\subsubsection
\let\subsection\subsubsection

The difference hierarchy was developed by Hausdorff, and a discussion can be found in \cite{Hausdorff44}.
We first cover the basic case, then develop the refinement we are considering and prove some basic lemmas.

Following the nomenclature of \cite{MoschovakisDST}, let us fix $\Gamma$ to be an \emph{adequate} pointclass closed under countable intersection in some ambient Polish space $\mathcal 
X$, for instance in the intended case, the class of $\ca$ sets in Baire space, $\bs$. Then, if $\alpha$ is a countable ordinal we denote the 
$\alpha$th level of the \emph{difference hierarchy} on $\Gamma$ by $\alpha\hyp\Gamma$: a pointset $A$ is in 
$\alpha\hyp\Gamma$ iff there is a sequence $\langle A_\xi\mid\xi\leq\alpha\rangle$, with each $A_\xi\in\Gamma$ and 
$A_\alpha=\varnothing$ such that:
\[
	x \in A \leftrightarrow \text{the least $\xi$ such that $x\notin A_\xi$ is odd,}
\]
in which case we say the sequence $\langle A_\xi\mid\xi\leq\alpha\rangle$ \emph{witnesses} that $A\in\alpha\hyp\Gamma$. Since $\Gamma$ is 
assumed to be closed under countable intersections we can take the sequence to be downwards-closed, that is 
$A_{\alpha}\supseteq A_\beta$ if $\alpha<\beta$, by rewriting $A_\beta$ as $\bigcap_{\alpha\leq\beta} A_\alpha$ if 
necessary. The definition of the difference hierarchy ensures that intersections with previous members of the sequence do 
not affect the resulting set $A$.

These pointclasses serve as a further stratification in the Borel hierarchy; each $\alpha\hyp\mathbf\Pi^0_n$ is strictly between 
$\mathbf\Pi^0_n$ and $\mathbf\Delta^0_{n+1}$ as long as $\alpha$ is countable, by a result of Hausdorff.
Note while we can consider pointclasses $\Gamma$ that are not closed under countable intersections, say $\mathbf\Sigma^0_n$, then 
already $\omega\hyp\mathbf\Sigma^0_n$ captures all of the $\mathbf\Pi^0_{n+1}$ sets, so this is not as interesting, and the results below 
become trivial.

We can make the hierarchy even finer by restricting the final set in the witnessing sequence. We adopt the following 
notation for this:
\begin{dfn}
	\label{dfn:refined-diff-hierarchy}
	If $\Lambda\subseteq\Gamma$ is a pointclass, we say $A\in(\alpha\hyp\Gamma)+\Lambda$ iff there is a sequence $\langle 
	A_\xi\mid\xi\leq\alpha+1\rangle$ witnessing that $A\in(\alpha+1)\hyp\Gamma$, with $A_\alpha\in\Lambda$. Note that 
	$A_{\alpha+1}$ is still $\varnothing$.
\end{dfn}
With this notation it is clear that $\alpha\hyp\Gamma+\Gamma$ is just $(\alpha+1)\hyp\Gamma$, but for general $\Lambda$ we 
end up with a different pointclass.


We now present a couple of elementary results about these refinements to the difference hierarchy.
Let $\neg\Gamma$ denote the dual class of the pointclass $\Gamma$.

\begin{lem}
	If $\lambda$ is a countable limit ordinal, then the dual class, $\neg((\lambda\hyp\Gamma)+\Lambda)=(\lambda\hyp\Gamma)+\neg\Lambda$.
\end{lem}
\begin{proof}
	Let $\langle A_\alpha\mid\alpha\leq\lambda+1\rangle$ witness that $A^c\in(\lambda\hyp\Gamma)+\Lambda$, and we want to 
	find a sequence witnessing that $A\in(\lambda\hyp\Gamma)+\neg\Lambda$. Define the sequence $\langle 
	B_\alpha\mid\alpha\leq\lambda+1\rangle$ as follows:
	\begin{align*}
		B_\eta &= \mathcal X, && \text{when $\eta<\lambda$ is zero or a limit ordinal}\\
		B_{\alpha+1} &= A_{\alpha}, && \text{when $\alpha\leq\lambda$}\\
		B_{\lambda} &= (A_{\lambda})^c,\\
		B_{\lambda+1} &= \varnothing.
	\end{align*}
	So assuming $\lambda>\omega$, the sequence looks like:
	\[
		\langle \mathcal X, A_1, A_2, \ldots, \mathcal X, A_{\omega}, \ldots, (A_\lambda)^c, \varnothing\rangle,
	\]
	with the whole space $\mathcal X$ inserted at zero and each limit position below $\lambda$. Now, this sequence is a 
	witness that some set is in $(\lambda\hyp\Gamma)+\neg\Lambda$, since it is of the correct length and 
	$B_\lambda\in\neg\Lambda$. Denote this set $B$ and we show $A^c=B$.
	
	Let $x\in\mathcal X$, then let $\alpha$ be the least ordinal such that $x\notin A_\alpha$. By construction, the least 
	$\beta$ such that $x\notin B_\beta$ is $\alpha+1$, unless $\alpha=\lambda+1$ in which case $\beta=\alpha-1$. In either 
	case $x\in A^c\leftrightarrow x\in B$.
	
	We've shown that $\neg((\lambda\hyp\Gamma)+\Lambda)\subseteq(\lambda\hyp\Gamma)+\neg\Lambda$, but note that we can apply the 
	same argument to get that $\neg((\lambda\hyp\Gamma)+\neg\Lambda)\subseteq(\lambda\hyp\Gamma)+\Lambda$, since 
	$\neg\neg A=A$.  Hence since taking dual classes preserves subsets, 
	$((\lambda\hyp\Gamma)+\neg\Lambda)\subseteq\neg((\lambda\hyp\Gamma)+\Lambda)$.
\end{proof}

For a pointclass $\Gamma$, $\Delta(\Gamma)$ denotes the self-dual pointclass, $\Gamma\cap\neg\Gamma$.

\begin{cor}
	If $\lambda$ is a countable limit ordinal, then $\lambda\hyp\Gamma+\Delta(\Gamma)\subseteq\Delta((\lambda+1)\hyp\Gamma)$.
\end{cor}
\begin{proof}
	This follows immediately from the previous lemma:
	\begin{align*}
		\lambda\hyp\Gamma+\Delta(\Gamma) = \lambda\hyp\Gamma+\neg\Delta(\Gamma)&=\neg(\lambda\hyp\Gamma+\Delta(\Gamma))
		\subseteq \neg((\lambda+1)\hyp\Gamma).
	\end{align*}
\end{proof}

Together these results will help extend the pattern down in the Borel hierarchy, that if we know $\Det(\Sigma^0_n)$, say, we get the dual class, $\Det(\Pi^0_n)$, for free.
Hence in the cases we are considering, having proved $\Det(\os+\Sigma^0_2)$, we will know immediately that $\Det(\os+\Pi^0_2)$, and indeed $\Det(\os+\Delta^0_2)$ hold.
\let\bibliography\oldbibliography
\let\section\oldsection
\let\subsection\oldsubsection

\subsection{Effective Descriptive Set Theory with Uncountable Spaces}
\label{sub:Effective Descriptive Set Theory in Uncountable Spaces}

\let\oldbibliography\bibliography
\def\bibliography#1{\relax}
\let\oldsection\section
\let\oldsubsection\subsection
\let\section\subsubsection
\let\subsection\subsubsection
The auxiliary game we construct to prove $\Det(\os+\Gamma)$ has a payoff set in the space $(\omega\times\aleph_\omega)^\omega$.
While it's not hard to see that such a set will be in the boldface counterpart of $\Gamma$ in this space, we will need a closer analysis than this.

Firstly, we need to know that the payoff set is at least definable over the weak models for which we have indiscernibles, and secondly proving determinacy in these models requires tools from effective descriptive set theory.
Effective descriptive set theory is defined in terms of computable functions on $\omega$, which will not suffice when working with an uncountable space.
We therefore need to generalise the theory and reprove several basic results in the uncountable context.

Let $T^*$ be the tree $F^{<\omega}$ for $F=\omega\times\aleph_\omega$.
This will be the tree for the auxiliary game.

\begin{dfn}[Generalised Recursive Pointclasses]
	\label{dfn:generalised-recursive}
	Define the following structure, the analogue of the hereditarily finite sets for classical recursion theory:
	\[
		\mathcal H=\langle L_{\aleph_\omega}[\langle\aleph_i\mid i<\omega\rangle],\in,\langle\aleph_i\mid 
		i<\omega\rangle\rangle.
	\]
	Fix a predicate $R(x^*,a)\subseteq [T^*]\times F$ (although the below definitions generalise trivially to predicates on $[T^*]^n\times F^m\times\omega^l$).
	Then $R$ is \emph{generalised-semi-recursive} if there is a $\Sigma_1(\mathcal H)$-definable set $X\subseteq T^*\times F$ such that
	\[
		R(x^*,a)\iff\exists m \langle x^*\upharpoonright m,a\rangle\in X.
	\]
	In other words, if $R$ is the $\Sigma_1(\mathcal H)$-union of basic open subsets of $[T^*]\times F$.
	$R$ is called \emph{generalised-recursive} if its complement is also of this form.
\end{dfn}
\begin{rem}
	Compare this with $HF$, over which the computably enumerable relations on $\omega$ are $\Sigma_1$, or to $\alpha$-recursion theory which considers subsets of $\alpha$.
	This can be seen as extending $\alpha$-recursion theory to subsets of $\alpha^\omega$, where we in particular take $\alpha=\aleph_\omega$.

	A coding of $(\aleph_\omega)^{<\omega}$ into $\aleph_\omega$ will allow us to talk about generalised-recursive relations on $T^*$ itself.
	Since $\mathcal H$ is closed under G\"odel pairing, the G\"odel pairing function $G$ on $\aleph_\omega$ is total and hence $\Delta_1^\mathcal H$, allowing consideration of such relations.
\end{rem}

\begin{dfn}[Generalised Kleene Pointclasses]
	Then for a predicate $P$ (which may be a subset of $[T^*]\times F^m\times\omega^l$ in general, but which we write below as a subset of $[T^*]\times F$ for clarity):
	\begin{enumerate}
		\item $P$ is $\wt\Sigma^0_1$ iff $P$ is generalised-semi-recursive;
		\item $P$ is $\wt\Sigma^0_{n+1}$ iff there is a $\wt\Pi^0_n$ predicate $R\subseteq [T^*]\times F\times\omega$ such that $P(x^*,a)\iff \exists b\in 
			\omega(R(x^*,a,b))$;
		\item $P$ is $\wt\Pi^0_n$ iff $\neg P$ is $\wt\Sigma^0_n$;
		\item $P$ is $\wt\Delta^0_n$ iff it is $\wt\Sigma^0_n$ and $\wt\Pi^0_n$.
	\end{enumerate}
	We then define the generalised analytic sets for $X\subseteq[T^*]\times F$:
	\[
		X\in\wt\Sigma^1_1 \iff \exists Y\in\wt\Pi^0_1(x\in X\leftrightarrow \exists y\in[T^*](\langle 
		x,y\rangle\in Y)).
	\]
	The $\wt\Sigma^1_n$, $\wt\Pi^1_n$ and $\wt\Delta^1_n$ sets follow as usual.
\end{dfn}

\begin{rem}
	The notion of computability is still the only difference with the usual lightface hierarchy; the subsequent levels are still built by (generalised) recursive unions, and complementation.
	Note that we take countable unions so as to align this hierarchy with the usual Borel hierarchy.

\end{rem}

\begin{lem}
	\label{lem:closed-under-recursive-functions}
	Suppose $P$ is $\wt\Sigma^0_n$ or $\wt\Pi^0_n$ and $f:[T^*]\to[T^*]$ is such that the relation $G\subseteq[T^*]\times T^*$ given by $G(a,p)\iff p \subseteq f(a)$.
	Then the set $f^{-1}`` P$ is also $\wt\Sigma^0_n$ or $\wt\Pi^0_n$, respectively.
\end{lem}
\begin{proof}
	First suppose $P$ is $\wt\Sigma^0_1$ and consider the set $f^{-1}`` P = \{a\mid f(a)\in P\}$.
	By our assumption on $f$ there is a $\Sigma_1(\mathcal H)$ set C such that
	\[
		p\subseteq f(a) \iff \exists m(\langle a\upharpoonright m, p\rangle\in C).
	\]
	Let $\bar C$ be the $\Sigma_1(\mathcal H)$ set $\{ q\mid\exists p\in P(\langle q,p\rangle\in C)\}$, and note that:
	\begin{align*}
		\exists m(a\upharpoonright m\in \bar C)
			&\iff \exists m(\exists p\in P(p \subseteq f(a)\upharpoonright m)\\
			&\iff \exists m(f(a)\upharpoonright m\in P)\\
			&\iff a\in f^{-1}``P,
	\end{align*}
	which is $\wt\Sigma^0_1$.

	If the result holds for $\wt\Sigma^0_n$, and $\neg Q = P \in\wt\Sigma^0_n$ then:
	\[
		a\in f^{-1}`` p \iff f(a)\in P\iff f(a)\notin Q\iff a\notin f^{-1}`` Q.
	\]
	So the result holds for $\wt\Pi^0_n$.

	Now assuming inductively that the result holds for $\wt\Pi^0_n$, let $P$ be $\wt \Sigma^0_{n+1}$.
	There is thus a $\wt\Pi^0_n$ set $Q$ such that:
	\begin{align*}
		f^{-1}``P &= \{a\mid \exists m(\langle f(a),m\rangle\in Q)\}\\
			&= \{a\mid \exists m(\langle a,m\rangle\in \bar f^{-1}``Q)\},
	\end{align*}
	Where $\bar f$ is the function given by $\bar f(\langle a, m\rangle)=\langle f(a), m\rangle$.
\end{proof}
The following two propositions are obvious relationships between these pointclasses and the usual pointclasses.

\begin{prop}
	\label{prop:relation-old-new-hierarchies}
	$\Sigma^0_n\subseteq\wt\Sigma^0_n\subseteq\mathbf\Sigma^0_n$, where the pointclasses are taken to be in the spaces of $\omega^\omega$, $[T^*]$ and $[T^*]$, respectively; the boldface Borel hierarchy on $[T^*]$ being the usual topological definition.
\end{prop}
\begin{proof}
	The first inclusion is seen by an easy induction starting with the observation that $\Sigma^0_1$ relations on $\omega^\omega$ are a union of a $\Sigma_1^{HF}$ set of basic open sets, which is thus also a $\Sigma_1(\mathcal H)$ set.
	The second inclusion follows directly from the definition.
\end{proof}

In the following proposition, $\vee$ and $\wedge$ denote the pointclasses formed by sets being the union (respectively intersection) of a set in the first class with one in the second.
\begin{prop}
	$(\Sigma^0_n \vee \wt\Sigma^0_1),(\Sigma^0_n \wedge \wt\Pi^0_1)\subseteq \wt\Sigma^0_n$ for $n>1$.
\end{prop}
\begin{proof}
	For the first we need to observe that $[T^*]$ is a metrizable space, hence $\wt\Sigma^0_1$ sets are $\wt\Sigma^0_n$. (The 
	proof works as in the classical setting.) By 
	the above, for the second part we just need to see that $\wt\Pi^0_1\subseteq\wt\Sigma^0_n$ for $n>1$ as for the usual 
	pointclasses.
\end{proof}

We note at this point that ultimately we will be interested in the analogue of $\mathcal H$ built on some sequence $\vec c\in[\aleph_\omega]^\omega$ instead of $\langle\aleph_i\mid i\in\omega\rangle$.
For concreteness we present this section in terms of the $\aleph_i$s, since this doesn't change anything.
For the rest of this section, by ``admissible'' we mean ``admissible in the language of set theory augmented with 
constants $\vec c$;'' for now they will be interpreted as the $\aleph_i$s but this need not be the case in general.
\begin{lem}
	\label{lem:can-rank-gr-trees}
	If $R$ is a generalised-recursive well-founded relation and $M$ is admissible with $\mathcal H\in M$, then 
	$\rk R\in M$.
\end{lem}
\begin{proof}
	If $R$ is $\wt\Delta^0_1$ and wellfounded, then let 
	\[
		S=\{\langle p_0,\ldots,p_k\rangle\mid\forall i<k(p_{i+1} R p_i)\}.
	\]
	For $p,q\in S$ let $p <_R q$ if either $p\supsetneq q$ or, for $i$ the least point where $p_i\neq q_i$, $p_i < q_i$.
	Then $<_R$ wellorders $S$, is generalised-recursive and $\ot <_R$ is at least $\rk R$.

	Now we need to show that the order-type of a generalised-recursive well\-order---a \emph{generalised-recursive ordinal}---is an element of $M$.
	Let $\alpha=On\cap M$, hence $\alpha>\aleph_\omega$ is an admissible ordinal and
	\[
		N=\langle L_\alpha[\langle\aleph_i\mid{i\in\omega}\rangle],\in,\langle\aleph_i\mid{i\in\omega}\rangle\rangle\rangle
	\]
	is admissible.
	Now, since $<_R$ is definable over $\mathcal H$, it is an element of 
	\[
		\langle L_{(\aleph_\omega)+1}[\langle\aleph_i\mid{i\in\omega}\rangle],\in,\langle\aleph_i\mid{i\in\omega}\rangle\rangle\rangle
	\]
	which is a subset of $N$.
	By admissibility (i.e. recursion) let $f$ be a $\Sigma_1^N$ order-preserving bijection between $<_R$ and $\ot(<_R)$.
	Hence $\ot(<_R)\in N$, and $On\cap M=On\cap N$.
\end{proof}

\begin{lem}
	\label{lem:generalised-tree-repn}
	Let $x^*\in[T^*]$ be $\wt\Sigma^1_1$. Then there is a $\wt\Pi^0_1$ tree $R\subseteq \omega\times F\times T^*$ 
	such that
	\[
		x^*(n)=a\iff \exists y\in [T^*] (\langle \langle n, a\rangle,y\rangle\in [R]).
	\]
\end{lem}
\begin{proof}
	A $\wt\Pi^0_1$ relation on $\omega^n\times F^m\times T^*$ is a tree if its $T^*$ component is closed under initial segments, that is:
	\[
		\langle i_1,\ldots i_n,a_1,\ldots,a_m,q\rangle\in R\wedge p\subseteq q \implies \langle i_1,\ldots i_n,a_1,\ldots,a_m,p\rangle\in R.
	\]
	Let $R$ be such a tree with $n=m=1$ for simplicity, then the set of branches of $R$ is:
	\[
		[R] = \{\langle n, a, x^*\rangle\in \omega\times F\times [T^*]\mid\forall m (\langle n,a,x^*\upharpoonright m \in R)\}.
	\]
	Now, being $\wt\Sigma^1_1$ means that there is a $\wt\Pi^0_1$ predicate $Y$ such that:
	\begin{align*}
		x^*(n) = a 
			&\iff \exists y\in [T^*] (Y(\langle \langle n,a\rangle,y\rangle))\\
			&\iff \exists y\in [T^*] \forall m(\langle\langle n,a\rangle,y\upharpoonright m\rangle\notin A_Y),
	\end{align*}
	where $A_Y$ is the $\Sigma_1(\mathcal H)$ set witnessing that $Y$ is $\wt\Pi^0_1$.
	So let $R$ consist of all elements $\langle\langle n,a\rangle, p\rangle$ such that $\forall m\leq|p| \langle\langle n,a\rangle,p\upharpoonright m\rangle\notin A_Y$, so that $R$ is a tree, is $\wt\Pi^0_1$ and $x^*(n) = a\iff \exists y 
	\langle\langle n,a\rangle,y\rangle\in [R]$.
\end{proof}

For $\langle n,a\rangle\in\omega\times F$ and a tree $R\subseteq \omega\times F\times T^*$, denote by $T_{\langle n,a\rangle}$ the $\langle n,a\rangle$'th part of $R$, that is, $\{p\in T^*\mid \langle\langle n,a\rangle,p\rangle\in 
R\}$.
Then the above lemma says that $x^*$ is $\wt\Sigma^1_1$ iff there is a $\wt\Pi^0_1$ tree $R$ such that $x^*(n)=a\iff R_{\langle n,a\rangle}$ is illfounded (the reverse direction being obvious).

Of course, the lemma extends easily to relations $X\subseteq F^n\times (T^*)^m\times [T^*]^l$, in which case $X$ is $\wt\Sigma^1_1$ if (and only if) there is a generalised-recursive tree $R\subseteq F^n\times F^m\times (T^*)^l \times T^*$ such that
\[
	\langle a, p, x\rangle\in X \iff \exists y\in [T^*] (\langle a, e(p), x, y\rangle\in [R])
\]
where $e$ is a fixed generalised-recursive coding of $F$ into $T^*$.
Put more succinctly, $X\subseteq[T^*]$ is $\wt\Sigma^1_1$ iff there is a generalized-recursive tree $R\subseteq T^*\times T^*$ such that:
\[
	x\in X\iff R_x \text{ is illfounded}
\]
where $R_x=\{y\in T^*\mid\langle x\upharpoonright |y|,y\rangle\in R\}$.

\begin{lem}
	If $x^*\in[T^*]$ is $\wt\Sigma^1_1$, then $x^*$ is $\Pi_1$-definable over any admissible containing $\mathcal H$ as an element.
\end{lem}
\begin{proof}
	This is essentially the Spector-Gandy theorem for elements of $[T^*]$ instead of $2^\omega$. Let $M$ be an arbitrary admissible set with $\mathcal H\in M$.

	By the above find $R\subseteq \omega\times F\times T^*$, a $\wt\Pi^0_1$ tree such that $x^*(n) = a\iff R_{\langle n,a\rangle}$ is illfounded.
	$R$ is $\Pi_1$ definable over $\mathcal H$, hence $\Delta_1$ definable over $M$, and so by $\Delta_1$ separation in $M$, $R\in M$.
	Let $\varphi(\langle n,a\rangle)$ be the $\Sigma_1$ formula:
	\[
		\exists\gamma,g(\gamma\in On\wedge g:R_{\langle n,a\rangle}\to \gamma\text{ is order-preserving}).
	\]
	Suppose $x^*(n)\neq a$. Then $R_{\langle n,a\rangle}$ is wellfounded and an element of $M$, so by admissibility there is 
	$\gamma<On\cap M$ and $g:R_{\langle n,a\rangle}\to\gamma$ witnessing the fact, hence $\varphi(\langle n,a\rangle)$ holds 
	in $M$. On the other hand if $M\vDash\varphi(\langle n,a\rangle)$ then $R_{\langle n,a\rangle}$ is really wellfounded
	and so $x^*(n)\neq a$. Hence $x^*(n) = a\iff M\vDash\neg\varphi(\langle n,a\rangle)$.
\end{proof}

\begin{cor}
	Under the same hypotheses, $x^*$ is an element of any admissible set which itself contains an admissible containing $T^*$.
\end{cor}
\begin{proof}
	Let $M\in N$ be admissible and $T^*\in M$. Then $x^*$ is definable over $M$ and $M\in N$ is transitive. Hence, since 
	the satisfaction relation is $\Delta_1^{\mathsf{KP}}$, by admissibility $(x^*)^M\in N$. Since $(x^*)^M$ is the ``true'' 
	$x^*$, we have that $x^*\in N$.
\end{proof}

\begin{rem}
	This conclusion remains true if $x^*$ is a $\wt\Sigma^1_1$ element of $2^{T^*}$, since it will be a $\Pi_1$-definable 
	class of $M$, hence an element of $N$.
\end{rem}

The following is then a weak analogue of the Kleene Basis Theorem in our setting:

\begin{thm}
	\label{thm:generalised-kbt}
	If $X^*\subseteq [T^*]$ is $\wt\Sigma^1_1$ and non-empty, then $X^*$ has an element definable over any admissible set $M$ with $T^*\in M$.
\end{thm}
\begin{proof}
	In light of the above, it will suffice to show that $X^*$ has an element which is recursive in some $\wt\Sigma^1_1$ element of $[T^*]$.
	This will then be definable over any $M$ which is admissible and contains $T^*$ as an element.

	Since $X^*$ is $\wt\Sigma^1_1$ there is a $\wt\Pi^0_1$ tree $T\subseteq T^*\times T^*$ such that $x^*\in X^*\iff T_{x^*}$ is illfounded.
	Hence any infinite branch of $T$ determines an element of $X^*$, and we need to find such a branch in $M$.
	Define $P$ to be the set of elements of $T$ which can be extended to an infinite branch:
	\[
		P=\{p=\langle (a_0,b_0),(a_1,b_1),\ldots,(a_t,b_t)\rangle\in T\mid \exists y\supseteq p\forall n(y\upharpoonright n 
		\in T)\}.
	\]
	$P$ is $\wt\Sigma^1_1$ by definition, so by the Spector-Gandy theorem for $T^*\times T^*$, we have that $P$ is 
	$\Pi_1^M$.
	
	Then the leftmost path through 
	$P$ is defined by recursion by:
	\[
		y(n) = \text{the least $a\in F\times F$ such that } (y\upharpoonright n)\concat a\in P
	\]
	which is definable from $P$ and hence over $M$, and the left part of $y$ is an element of $X^*$.
\end{proof}

\begin{rem}
	It would be natural to consider the more direct analogue of Kleene's Basis Theorem by finding a $\wt\Sigma^1_1$ basis for the $\wt\Sigma^1_1$ sets, but for our purposes, the above suffices and is more expeditious.
\end{rem}

We also note that Shoenfield Absoluteness holds for these models:

\begin{thm}
	\label{thm:generalised-shoenfield}
	If $A\subseteq\bs$ is $\Sigma^1_2(a)$, then $A$ is absolute for models $M$ as considered above, i.e. which are admissible with $T^*\in M$, if in addition $a\in M$.
\end{thm}
\begin{proof}
	Note we are now concerned with the usual $\Sigma^1_2$ sets, so the crucial things to know are that firstly $\Sigma^1_2(a)$ sets are $\omega_1$-Suslin, and secondly that the Shoenfield tree can be ranked in our models.
	(See, for instance, the final remarks in \cite{Kanamori}, 13.15.)
	Thankfully, the Shoenfield tree is easily generalised-recursive; it is defined as:
	\[
		\hat T = \{ \langle p, u\rangle\in \omega^{<\omega}\times\omega_1^{<\omega}\mid \forall i,j <|p| (s_i\supseteq s_j \wedge \langle s\upharpoonright |s_i|, s_i\rangle\in T \rightarrow u(i)<u(j) \}
	\]
	where $\langle s_i\mid i\in\omega\rangle$ is a standard recursive enumeration of elements of $\omega^{<\omega}$ and $T$ is the $\Pi^1_1(a)$ tree witnessing that $A$ is $\Sigma^1_2(a)$.
	Now this is a $\Delta_1(\{a\})$-definable element of $L_{\omega_1}[a]$, and is
	hence $\Delta_1^{\mathcal H[a]}$ where
	\[\mathcal H[a] = \langle L_{\aleph_\omega}[\langle\aleph_i\rangle,a],\in,\langle\aleph_i\rangle,a\rangle\]
	Hence $\hat T$ is also $\Delta_1^{\mathcal H[a]}$ and with minor modification, the method of Lemma \ref{lem:can-rank-gr-trees} shows that wellfoundedness of such trees is absolute for $M$.
	But this is just what is needed to prove Shoenfield's absoluteness theorem.
\end{proof}
\let\bibliography\oldbibliography
\let\section\oldsection
\let\subsection\oldsubsection

\section{Obtaining Indiscernibles}
\label{sec:Obtaining Indiscernibles}

\let\oldbibliography\bibliography
\def\bibliography#1{\relax}
\let\oldsection\section
\let\oldsubsection\subsection
\let\section\subsection
\let\subsection\subsubsection

\section{Introduction}
We aim to define a notion of Prikry forcing which, when executed over a model of \nkp yields a generic extension which is also a model of \nkp, and which satisfies a form of \emph{generating indiscernibility}.
This is all a straightforward checking that the forcing can be defined over the model and preserves \nkp. First we define the indiscernibility principle we are interested in:

\begin{dfn}
	\label{dfn:generating-indiscernibles}
	A closed-unbounded class of ordinals $C$ is a class of \emph{$\Sigma_n$ generating indiscernibles} for the theory $T$ if, for any $\vec c,\vec d\in[C]^\omega$, letting $\mathcal A_T[\vec c]$ be the least transitive model of $T$ containing $\vec c$ as an element, we have $\mathcal A_T[\vec c]\equiv_{\Sigma_n}\mathcal A_T[\vec d]$.
\end{dfn}

First of all we seek to prove:

\begin{thm}
	\label{thm:nkp-mouse-implies-gamma-n}
	If there exists a non-trivial mouse which is a model of \nkp, then there exists a class of $\Sigma_n$ generating indiscernibles for \nkp.
\end{thm}

Fix a mouse
\[
	M=\langle J^E_\theta, E, F\rangle\vDash\text{``\nkp + $F$ is a normal measure on $\kappa$''}.
\]
Let us fix some terminology for the remainder of the section:

\begin{dfn}
	\label{def:M-lambda}
	Let $C=\{\kappa_0,\kappa_1,\ldots\}$ be the class of iteration points of $M$ by $F$, and let $\vec c=\{c_0,c_1,\ldots\}\in[C]^\omega$ be any fixed $\omega$-sequence of them indexed in increasing order.
	Then let $\lambda$ be such that the measurable cardinal of $M_\lambda$, the $\lambda$th iterate, is $\sup\vec c$, $\pi:M\to M_\lambda$ the iteration map.
	With $E_\lambda, F_\lambda$ the filter sequence and top measure of $M_\lambda$, respectively, let $\widetilde\theta$ be such that $M_\lambda=\langle J^{E_\lambda}_{\widetilde\theta},E_\lambda,F_\lambda\rangle$.

\end{dfn}

Now note that by $\Sigma_n$ separation we may apply \L os' theorem to $\Sigma_n$ formul\ae, so these ultrapowers preserve $\Sigma_n$ sentences, and hence $C$ is a set of $\Sigma_n$ indiscernibles for $M_\lambda$.

We use the notation $X\smallsetminus p$, for sets of ordinals $X,p$ to mean $X\setminus(\sup p + 1)$.
\begin{lem}
	\label{lem:contains-all-indiscernibles}
	Let $p=\{ c_0,\ldots,c_k\}$ and $e\in[C]^{<\omega}$ be a finite set of iteration points with $\max e < c_k$.
	Let $X\in F_\lambda$ be $\Sigma_n$-definable in $M_\lambda$ from $\pi(f)$ (for $f\in M^\kappa$), $e, p$.
	Then $X\supseteq C\smallsetminus p$.
\end{lem}
\begin{proof}
	By the properties of iterated ultrapowers, the measure 1 set $X$ must include an end-segment from $C$, which, as remarked above, is a set of $\Sigma_n$-indiscernibles for $M_\lambda$.
	Thus $X$ contains at least one indiscernible greater than all the indiscernible parameters in its definition, and so by indiscernibility, it contains all indiscernible greater than $\max p$.
\end{proof}

We need to know more, namely that $\nkp$ is preserved in the ultrapower.

\begin{thm}
	\label{thm:iterating-preserves-separation}
	$M_\lambda\vDash\nkp$.
\end{thm}
\begin{proof}
	First suppose that we know that the iterate $M_\alpha\vDash\nkp$.
	We use the standard fact that the theory of \nkp is equivalent to there being no $\Sigma_n$-definable partial function, whose image of a set in the model is unbounded in the ordinals of the model.

	Thus let $\varphi(\xi,\gamma,p)$ define a $\Sigma_n$-partial function $f$; $f(\xi)=\gamma\leftrightarrow\varphi(\xi,\gamma,p)$, with $f$ and $A\in M_{\alpha+1}$ witnessing the failure of \nkp, i.e. $f``A$ is unbounded in $On\cap M_{\alpha+1}$.

	Expanding $f``A$ being unbounded, we get the $\Pi_{n+1}$ formula:
	\[
		M_{\alpha+1}\vDash\forall\tau \exists\gamma>\tau \exists\xi\in A (\varphi(\xi,\gamma,p)).
	\]
	Fix $p=\pi_{\alpha\alpha+1}(g)(\kappa_\alpha)$ for some $g\in M_\alpha$.
	By cofinality of the ultrapower embedding $\pi_{\alpha\alpha+1}$, let $\bar A$ be such that $A\subseteq\pi_{\alpha\alpha+1}(\bar A)$.
	Then fix an arbitrary $\bar\tau\in M_\alpha$, and by applying \L o\'s' theorem to the above for $\tau=\pi_{\alpha\alpha+1}(\bar\tau)$:
	\[
		M_\alpha\vDash \{\beta<\kappa_\alpha\mid \exists\gamma>\bar\tau\exists\xi\in\bar A(\varphi(\xi,\gamma,g(\beta)))\}\in F^\alpha.
	\]
	Hence:
	\[
		M_\alpha\vDash\forall\bar\tau[\exists q\exists\xi\in\bar A\exists\gamma>\bar\tau(\varphi(\xi,\gamma,q))]
	\]
	and hence the formula $\psi(\xi,\gamma) \equiv \gamma=\sup \{\gamma' \mid \exists q\varphi(\xi,\gamma',q)\}$ defines a $\Sigma_n$ partial function $\bar f$, since the $\gamma'$s are bounded by \nkp.
	But then:
	\[
		\bar f``\bar A = \{\gamma\mid\exists\xi\in\bar A(\psi(\xi,\gamma))\}
	\]
	is unbounded, thus contradicting \nkp in $M_\alpha$.

	Thus we have shown that if $M_\alpha\vDash\nkp$, so too does $M_{\alpha+1}$, and we just have to show that the limit models behave similarly.
	But if $\nu$ is limit and $p,X\in M_\nu$ then take $\theta$ such that $p=\pi_{\theta,\nu}(\bar p)$ and $X=\pi_{\theta,\nu}(\bar X)$.
	By induction we may assume that if $\varphi$ is $\Sigma_n$, $\bar B=\{\alpha\in\bar X\mid\varphi(\alpha,\bar p)\}$ is an element of $M_\theta$.
	Hence
	\[
		\{\alpha\in X\mid\varphi(\alpha,p)\}
		= \{\alpha\in \pi_{\theta,\nu}(\bar X)\mid\varphi(\alpha,\pi_{\theta,\nu}(\bar p))\}
		= \pi_{\theta,\nu}(\bar B)\in M_\nu
	\]
	by elementarity, and we are done.
\end{proof}

\section{The Forcing}

$M_\lambda$ will be the ground model for Prikry forcing, which will add the set of indiscernibles $\vec c$ to the generic extension, whilst preserving \nkp.

Rowbottom's theorem will be needed for several of the proofs in this section:

\begin{lem}
	$\mathsf{KP}$ proves that if $\kappa$ is a measurable cardinal with normal measure $F$, and $f$ is a partition of 
	$[\kappa]^{<\omega}$ into less than $\kappa$ pieces, then there is a $\Delta_1(f)$ set $H\in F$ homogeneous for $f$, 
	i.e. such that for each $n$, $f$ is constant on $[H]^{n}$.
\end{lem}
\begin{proof}
	This is essentially \cite{Jech}, Theorem 10.22. Let $F$ be a normal measure on $\kappa$, and let $f$ partition 
	$[\kappa]^{<\omega}$ into less than $\kappa$ pieces.  If we can find measure one sets $H_n$ such that $f$ is 
	homogeneous on $[H_n]^n$, then $f$ is also homogeneous for $H=\bigcap H_n$.

	For $n=1$ consider the intersection $\bigcap\{\kappa\setminus f^{-1}``\{\alpha\}\mid \alpha\in\ran f\}$.
	This set is empty, so by $\kappa$-completeness one of the terms must have been measure zero, i.e. some $\alpha$ must have pre-image of measure one, which we take to be $H_1$.

	Now proceed by induction: let $f:[\kappa]^{n+1}\to I$ for some $|I|<\kappa$. Now for each 
	$\alpha<\kappa$ we define $f_\alpha$ with domain $[\kappa\setminus\{\alpha\}]^n$ by $f_\alpha(x)=f(\{\alpha\}\cup 
	x)$. By hypothesis there exist $\Delta_1$ sets $X_\alpha\in F$ such that $f_\alpha$ is constant on $[X_\alpha]^n$. We 
	fix this constant value to be $i_\alpha$.
	By \KP the sequence $\{X_\alpha\mid\alpha<\kappa\}$ exists.

	Let $X$ be the diagonal intersection of the $X_\alpha$s, 
	$\{\alpha<\kappa\mid\alpha\in\bigcap_{\gamma<\alpha}X_\gamma\}$, which is in $F$ by normality, and is a $\Delta_1$ 
	subset of $\kappa$. Now if $\gamma<\alpha_1<\ldots<\alpha_n$ are in $X$, then 
	$\{\alpha_1,\ldots,\alpha_n\}\in[X_\gamma]^n$, so that 
	$f(\{\gamma,\alpha_1,\ldots,\alpha_n\})=f_\gamma(\{\alpha_1,\ldots,\alpha_n\})=i_\gamma$. Now, $\gamma\mapsto 
	i_\gamma$ constitutes a partition of $[X]^n$ into $\gamma<\kappa$ pieces, so again by hypothesis we know there is
	$i\in I$ and a $\Delta_1$ set $H\subseteq X$ in $F$ such that $i_\gamma=i$ for all $\gamma\in H$, hence $f(x)=i$ for 
	all $x\in[H]^{n+1}$.
\end{proof}

The way we will use Rowbottom's theorem is, if we have a set $A\subseteq [\kappa]^{<\omega}$, we view the characteristic 
function of $A$ as a partition of $[\kappa]^{<\omega}$ into 2. There is then a $\Delta_1(A)$ set $Z\in F$ such that either 
$[Z]^{<\omega}\subseteq A$ or $[Z]^{<\omega}\cap A=\varnothing$. 

We now set up the basic forcing definitions in a way suitable for this purpose.
This will require a return to ramified forcing, as originally conceived by Cohen (see \cite{Cohen}) and used in \cite{Welch96}.
Since the forcing we use will be a class forcing from the perspective of the model, our setting is more general than that used by Cohen in \cite{Cohen}, and the use of theories stronger than \KP means we extend Welch's work, too.
In addition, accounts since Cohen's of ramified forcing have mostly been sketched, without full definitions or proofs, and the notation of Cohen's account is now outdated.
For these two reasons, we set out the full definitions and proofs here.
We will need to check that the usual properties of forcing hold in spite of our weak setting and the fact that we are using a class forcing.

\begin{dfn}
	Let $\langle\P,\leq\rangle$ be Prikry forcing defined over $M_\lambda$, that is:
	\begin{align*}
		\mathbb{P} :=\; & \{\langle p, X\rangle\mid p\in[\kappa]^{<\omega}, X\in F_\lambda\cap M_\lambda\}\\
		\langle p, X\rangle \leq \langle q, Y\rangle \:\leftrightarrow\; &
		q \text{ is an initial segment of } p \wedge X\cup(p\setminus q)\subseteq Y.
	\end{align*}
\end{dfn}
This class is $\Delta_1$-definable over $M_\lambda$, but since $F_\lambda$ is not a set in the model, it is not a set forcing and so most of the work will be towards proving that \nkp is preserved.

\begin{dfn}
	We define a ranked forcing language, $\LF$.
	First of all, we define constants for the language, which are representations of class terms and are intended to name sets in the generic extension.

	$\mathcal C_0$ consists of $\{\csym\}\cup\kappa_\lambda$, where $\csym$ is a constant symbol.

	If $\mathcal C_\alpha$ is defined then $\mathcal C_{\alpha+1}$ consists of the class terms $\{x^\alpha\mid\varphi(x^\alpha)\}$ where $\varphi$ is a formula built up from: symbols $\in, =, \csym$; any element of $\mathcal C_\beta$ for $\beta<\alpha$; logical connectives and where any quantified variable is of the form $\exists x^\beta$ for $\beta\leq\alpha$.

	Take unions at limit stages, and let $\mathcal C=\bigcup_{\alpha<\wt\theta} C_\alpha$. (Recall that $\wt\theta$ is the height of $M_\lambda$.)

	The \emph{rank} of one of these constants is the ordinal $\alpha$ such that it is an element of $\mathcal C_\alpha$.

	$\LF$ then consists of all formul\ae{} built from the symbols: $\in, =$; unranked variables $x, y, z,\ldots$; ranked variables $x^\alpha, y^\alpha, z^\alpha$ for $\alpha<\widetilde\theta$; connectives; quantifiers; and elements of $\mathcal C_\alpha$ for $\alpha<\widetilde\theta$.

	If $\varphi\in\LF$ then we say it is \emph{ranked} if every variable in it is ranked, in which case its rank is the maximum of the ranks of quantified variables in $\varphi$ and of the ranks of any constant terms occurring in $\varphi$.
\end{dfn}

We now define the weak forcing relation $\mathbf p\Vdash^*\varphi$ for $\mathbf p=\langle p, X\rangle\in\P$ and $\varphi$ a sentence in $\LF$:

\begin{enumerate}
	\item $\mathbf p\Vdash^* x\in y$ iff $x\in \mathcal C_\alpha, y\in \mathcal C_\beta$ and:
		\begin{enumerate}
			\item $\alpha = \beta = 0$ and either $x\in y\in\kappa$ or $x\in p\wedge y=\csym$; or
			\item $\alpha < \beta$, $y=\{z^\gamma\mid\varphi(z^\gamma)\}$ and $\mathbf p\Vdash^*\varphi(x)$; or else
			\item $\alpha\geq\beta$ and $\exists z\in \mathcal C_\gamma$ for some $\gamma$, either $\beta>\gamma$ or $\beta=\gamma=0$ and
				\[ \mathbf p\Vdash^*z=x\wedge z\in y \]
		\end{enumerate}
	\item $\mathbf p\Vdash^* x=y$ iff $\mathbf p\Vdash^*\forall z^\alpha(z\in x\leftrightarrow z\in y)$ for $\alpha$ the maximum of the ranks of $x$ and $y$.
	\item $\mathbf p\Vdash^* \varphi\wedge\psi$ iff $\mathbf p\Vdash^*\varphi$ and $\mathbf p\Vdash^*\psi$.
	\item $\mathbf p\Vdash^* \neg\varphi$ iff $\forall\mathbf q\in\P(\mathbf q\leq \mathbf p \implies \mathbf q\nVdash^*\varphi)$.
	\item $\mathbf p\Vdash^* \exists x^\alpha(\varphi(x^\alpha))$ iff there is some $t\in \mathcal C_\alpha$ such that $\mathbf p\Vdash^*\varphi(t)$.
	\item $\mathbf p\Vdash^* \exists x(\varphi(x))$ iff there is some $t\in \bigcup_\alpha \mathcal C_\alpha$ such that $\mathbf p\Vdash^*\varphi(t)$.
\end{enumerate}

We note that another notion of rank may be defined for formul\ae{} of the forcing language that are ranked as defined above, so that whenever the definition of $\mathbf p\Vdash^*\varphi$ refers to a formula $\psi$ of the forcing language, this new rank strictly decreases.
This makes the above recursive definition legitimate, and allows arguments about the forcing relation ``by induction on $\varphi$.''

\begin{dfn}
	A filter $G\subset \P$ is \emph{$M_\lambda$-generic} for \P if it meets all $\Sigma_n^{M_\lambda}$-definable subsets of $\P$ \emph{and} if, for every $\Sigma_n$ sentence $\varphi$ of the forcing language, there is a $\mathbf p\in G$ such that $\mathbf p\Vdash^*\varphi\vee\mathbf p\Vdash^*\neg\varphi$.
\end{dfn}
\begin{rem}
	Note that the latter requirement also implies that such a generic decides every $\Pi_n$ sentence.

	Usually these two forms of genericity are equivalent, but here we will have to prove them both.
\end{rem}

If $G$ is $M_\lambda$-generic for \P then let $\vec c=\bigcup\{p\in[\kappa_\lambda]^{<\omega}\mid\exists X\in F_\lambda(\langle p,X\rangle\in G)\}$.
If we define an appropriate $G$, this will be the sequence $\vec c$ of iteration points we picked above.
The generic extension is defined as $M_\lambda[G]=\langle J_{\widetilde\theta}^{\vec c}, \vec c\rangle$, and if there is only one $G$ under discussion we write $M_\lambda[\vec c]$.

This forcing relation obeys the usual properties:
\begin{lem}
	\leavevmode
	\begin{enumerate}
		\item For no $\mathbf p,\varphi$ does $\mathbf p\Vdash^*\varphi\wedge\neg\varphi$;
		\item If $\mathbf p\Vdash^*\varphi$ and $\mathbf q\leq\mathbf p$ then $\mathbf q\Vdash^*\varphi$;
		\item $\mathbf p\Vdash^*\varphi\iff\forall\mathbf q\leq \mathbf p\exists\mathbf r\leq\mathbf q(\mathbf r\Vdash^*\varphi)$;
	\end{enumerate}
\end{lem}
\begin{proof}
	\leavevmode
	\begin{enumerate}
		\item[1--2] are standard and can be found in \cite{Cohen} Section IV.4.
		\item[3.] The forward implication is obvious by 2 and the backwards implication is proved by induction on $\varphi$:
			Suppose first that $\varphi$ is $x\in y$.
			If $x,y\in \mathcal C_0$ and $y\in\kappa$, then, since forcing this doesn't depend on $\mathbf p$, there's nothing to do.
			If $\varphi$ is $x\in\csym$ then we have to show that $x\in p$ where $\mathbf p=\langle p, X\rangle$.
			But if not, then we can find $q\in X\setminus\{x\}$ and then $\langle p\cup\{q\},X\rangle\Vdash^*\neg\varphi$ for a contradiction.

			For every other case we use the induction hypothesis in the obvious way.
	\end{enumerate}
\end{proof}

The reason for using the ranked forcing language is that it allows the forcing relation to be simply definable, as we prove in the following lemma:

\begin{lem}
	\label{lem:prikry}
	If $\mathbf p\in\P$ and $\varphi\in\mathcal L_{F_\lambda}$ is a ranked sentence, then:
	\begin{enumerate}
		\item If $\mathbf p=\langle p,X\rangle$ there is a $\Delta_1^{M_\lambda}(\mathbf p,\varphi)$ set $Y\in F_\lambda$ such that $\langle p, Y\rangle\|\varphi$, that is either $\langle p,Y\rangle\Vdash^*\varphi$ or $\langle p, Y\rangle\Vdash^*\neg\varphi$; and
		\item $\langle p,Y\rangle\Vdash^*\varphi$ is $\Delta_1^{M_\lambda}$.
	\end{enumerate}
\end{lem}
\begin{proof}
	The first assertion is a re-statement of the Prikry lemma, but we will prove both statements in a simultaneous induction on the complexity of $\varphi$.

	First suppose $\varphi$ is $x\in y$ for $x,y\in \mathcal C_0$.
	Then certainly $\mathbf p\Vdash^*\varphi$ is $\Delta_1$-definable.
	If $\mathbf p\Vdash^*\varphi$, then set $Y=X$ and we are also done with part 2.
	If not, then:
	\[
		(x\notin y\vee y\notin\kappa)\wedge(x\notin p\vee y\neq\csym).
	\]
	This is preserved when strengthening $\mathbf p$, so in fact $\mathbf p\Vdash^*\neg\varphi$, and again we can take $Y=X$ for part 2.

	Now suppose we have proved the statement of the lemma for all formul\ae{} occurring in the definition of $\mathbf p\Vdash^*\varphi$.
	We shall not carry out the proof in every case, but the following illustrate the most important and involved parts:
	\begin{enumerate}
		\item If $\varphi$ is $x\in y$ with $x,y$ having rank $\alpha<\beta$ respectively, then let $y=\{z^\gamma\mid\psi(z^\gamma)\}$.
			Then the induction hypothesis says that for each $d\in \mathcal C_\gamma$ there is an extension of $\mathbf p$ to $\langle p, Y_d\rangle$ which decides $\psi(d)$.
			The condition $\mathbf p_x=\langle p, Y_x\rangle$ therefore decides $\psi(x)$.
			By definition, $\mathbf p_x\Vdash^*x\in y\iff \mathbf p_x\Vdash^*\psi(x)$, so we are done.
		\item If $\varphi$ is $\neg\psi$, then by hypothesis extend $\mathbf p$ to $\langle p, Y\rangle$ in a $\Delta_1$ way such that $\langle p, Y\rangle\|\psi$.
			But then $\langle p, Y\rangle$ also decides $\varphi$.

			Likewise if we know $\langle p,Y\rangle\Vdash^*\psi$ then $\langle p,Y\rangle\nVdash^*\varphi$ and vice-versa, and if the former is $\Delta_1$ definable over $M_\lambda$ then so is the latter.
		\item If $\varphi$ is $\exists x^\alpha(\psi(x^\alpha))$ then define:
			\[
				S^+ = \{q\in [X]^{<\omega}\mid\exists d,\exists X^d\subseteq X(\langle p\cup q, X^d\rangle\Vdash\psi(d))\}.
			\]
			By hypothesis, for each $q,d$ there is a $\Delta_1^{M_\lambda}$-definable set $Y$ such that $\langle p\cup q,Y\rangle\|\psi(d)$, and such that $\langle p\cup q,Y\rangle\Vdash^*\psi(d)$ is $\Delta_1$.
			This means that the existence of $X^d$ in the definition of $S^+$ is equivalent to $\langle p\cup q,Y\rangle\Vdash^*\psi(d)$ since, if $\langle p\cup q,Y\rangle\Vdash^*\neg\psi(d)$ there can be no such $X^d$ as any $\langle p\cup q,X^d\rangle$ is compatible with $\langle p\cup q, Y\rangle$.

			Thus $S^+$ is a $\Sigma_1^{M_\lambda}$ subset of $[\kappa_\lambda]^{<\omega}$ and by Rowbottom's theorem there is a set $A\in F_\lambda\cap M_\lambda$ such that either $[A]^{<\omega}\subseteq S^+$ or $[A]^{<\omega}\cap S^+=\varnothing$.

			Suppose $\langle p, A\rangle$ does not decide $\varphi$.
			Hence there are conditions $\langle r, Y\rangle, \langle s, Z\rangle$, each stronger than $\langle p, A\rangle$ such that one forces $\varphi$ and the other forces $\neg\varphi$.
			Extending the shorter sequence if necessary, assume $\ell h(r) = \ell h(s) = k$.
			Then $\langle r, Y\rangle\leq\langle p, A\rangle$ so $r\in[A]^k$, but similarly $s\in [A]^k$.
			Hence $[A]^k$ contains elements in $S^+$ and out of it.
			This is a contradiction and so $\langle p, A\rangle$ decides $\varphi$.
	\end{enumerate}
\end{proof}

\begin{cor}
	The relation $\mathbf p\Vdash^*\varphi$ is $\Delta_1^{M_\lambda}$ definable for ranked sentences $\varphi$, $\Sigma_n^{M_\lambda}$ for $\Sigma_n$ sentences and $\Pi_n^{M_\lambda}$ for $\Pi_n$ sentences of the forcing language.
\end{cor}
\begin{proof}
	Let $\mathbf p=\langle p,X\rangle$.
	If $\varphi$ is ranked then for each $q\in[X]^{<\omega}$ let $Y^q$ be the set produced by the above lemma given $\langle p\cup q,X\rangle, \varphi$.
	Then:
	\begin{align*}
		\mathbf p\Vdash\varphi
			&\iff \forall q\in[X]^{<\omega} (\langle p\cup q, Y^q\rangle\Vdash\varphi).
	\end{align*}
	The forward direction is obvious since $\langle p\cup q, Y^q\rangle\leq\mathbf p$, and the backward direction holds because then, taking any $X'\subseteq X$ we have that $\langle p\cup q,X'\cap Y^q\rangle\Vdash\varphi$, i.e. the set of extensions of $\mathbf p$ which force $\varphi$ is dense.

	This proves the first part, since the given formula is $\Delta_1$, and for unranked formul\ae{} the result is a simple induction on L\'evy rank.
\end{proof}

When interpreting a formula of \LF in $M_\lambda[\vec c]$, we interpret constant symbols from $\mathcal C_\alpha$ as the corresponding class term, $\csym$ as the predicate $\vec c$ and quantification over ranked variables $\exists x^\alpha$ as bounded quantification $\exists x\in J^{\vec c}_\alpha$.

\begin{lem}[Truth Lemma]
	If $G$ is $M_\lambda$ generic for $\P$ and $\varphi$ is a $\Sigma_n$ or a $\Pi_n$ sentence of \LF, then $M_\lambda[G]\vDash\varphi\iff\exists\mathbf p\in G(\mathbf p\Vdash^*\varphi)$.
\end{lem}
\begin{proof}
	Again this is proved by induction on $\varphi$, first for ranked formul\ae.
	\begin{enumerate}
		\item
			If $\varphi$ is $x\in y$, we must check several cases, firstly assuming that $\mathbf p\in G,\mathbf p\Vdash^*\varphi$.
			If $x,y$ in $\mathcal C_\alpha,\mathcal C_\beta$, respectively, then firstly consider the case when $\alpha=\beta=0$.
			Here we only need to check that $\mathbf p\Vdash^*x\in\csym\implies M_\lambda[\vec c]\vDash x\in\vec c$, which is true since by definition $x\in p$, and $\mathbf p\in G$ so $p\subseteq\vec c$.

			If $\alpha<\beta$ then $\mathbf p\Vdash^*\psi(x)$ for $\psi$ the defining formula of $y$.
			By induction, $M_\lambda[\vec c]\vDash\psi(x)$, so $M_\lambda[\vec c]\vDash x\in y$.

			If $\alpha\geq\beta$ then $\mathbf p\Vdash^*z=x\wedge z\in y$ for some $z\in \mathcal C_\gamma$, so
			inductively, $M_\lambda[\vec c]\vDash z=x\wedge z\in y$ which completes the forward direction.

			If $M_\lambda[\vec c]\vDash x\in y$ then first suppose $y=\vec c$.
			Then $\exists\langle p,X\rangle\in G(x\in p)$, i.e. $\langle p, X\rangle\Vdash^*x\in y$.
			If $y$ is an ordinal then every condition will force $x\in y$.
			Otherwise, find $\alpha,\beta$ and $\bar x\in \mathcal C_\alpha, \bar y\in \mathcal C_\beta$ such that $\bar x,\bar y$ evaluate to $x, y$, respectively, in $M_\lambda[\vec c]$.
			When $\alpha<\beta$, the proof is straightforward by induction.
			If $\alpha\geq\beta$, we need to find $\gamma<\beta$, $\bar z\in \mathcal C_\gamma$ such that $\mathbf p\Vdash^*\bar z=\bar x\wedge\bar z\in\bar y$.
			But if $x\in y$ then $x$ occurs in the $J^{\vec c}$ hierarchy strictly before $y$, and so is realised as a class term in $\mathcal C_\gamma$ for $\gamma<\beta$.
			From here induction gives us what we want.

		\item If $\varphi$ is $x=y$ then
			\begin{align*}
				M_\lambda[\vec c]\vDash x=y
				&\iff\exists\alpha<\widetilde\theta(\forall z\in J_\alpha^{\vec c}(z\in x\leftrightarrow z\in y))\\
				&\iff\exists\mathbf p\in G(\mathbf p\Vdash^*(\forall z^\alpha(z\in x\leftrightarrow z\in y)))\\
				&\iff\exists\mathbf p\in G(\mathbf p\Vdash^*(x=y)),
			\end{align*}
			where we use the induction hypothesis in the second equivalence.

		\item If $\varphi$ is $\psi_0\wedge\psi_1$ then the argument is standard.

		\item
			If $\varphi$ is $\neg\psi$ and $\mathbf p\in G$, $\mathbf p\Vdash^*\neg\psi$, then suppose $\psi$ held in $M_\lambda[\vec c]$.
			Then, by the induction hypothesis there is $\mathbf q\in G(\mathbf q\Vdash^*\psi)$.
			But there would then be a $\mathbf r\leq\mathbf p,\mathbf q$ with $\mathbf r\Vdash^*\psi\wedge\neg\psi$, which is impossible, completing the proof of the right-to-left direction.

			On the other hand if $M_\lambda[\vec c]\vDash\neg\psi$ then there is by genericity at least a condition $\mathbf p\in G$ such that $\mathbf p\|\psi$.
			If $\mathbf p\Vdash^*\neg\psi$ then we are done, and otherwise the inductive hypothesis for $\psi$ would imply that $M_\lambda[\vec c]\vDash\psi$, which is impossible.

		\item If $\varphi$ is $\exists x^\alpha\psi(x^\alpha)$ and $\mathbf p\in G$, $\mathbf p\Vdash^*\varphi$ then there is some $x\in \mathcal C_\alpha$ such that $\mathbf p\Vdash^*\psi(x)$, and hence $M_\lambda[\vec c]\vDash\psi(x)$, with $x$ interpreted appropriately, hence $M_\lambda[\vec c]\vDash\exists x\in J^{\vec c}_\alpha(\psi(x))$.
			The opposite direction is identical.
		\item If $\varphi$ is $\exists x\psi(x)$ and $\mathbf p\in G$, $\mathbf p\Vdash^*\varphi$ then there is some $\alpha$ such that $\mathbf p\Vdash^*\exists x^\alpha\psi(x^\alpha)$, and we can use the previous case.
			If $M_\lambda[\vec c]\vDash\exists x\psi(x)$ then, by the definition of the generic extension, $M_\lambda[\vec c]\vDash x\in J_\alpha^{\vec c}(\psi(x))$, and again we proceed as above.

	\end{enumerate}
\end{proof}

We adopt the usual generalisation of diagonal intersection:
\[
	\Delta_p A_p = \Delta_\alpha\left(\bigcap\{A_p\mid\max p<\alpha\}\right).
\]

\begin{thm}
	The filter
	\[
		G_{\vec c} = \{\langle p, X\rangle\in\mathbb P\mid p\text{ is an initial segment of $\vec c$ and } \vec c \setminus 
			p\subseteq X\}
	\]
	is $M_\lambda$-generic for $(\P)^{M_\lambda}$.
\end{thm}
\begin{proof}
	We first prove that $G_{\vec c}$ intersects every $\Sigma_n$ dense class of $M_\lambda$.
	Let $D$ be a $\Sigma_n^{M_\lambda}$ subset of \P.
	Let $S_p=\{q\in[\kappa_\lambda]^{<\omega}\mid\exists X(\langle p\cup q, X\rangle\in D)\}$.
	This is $\Sigma_n$, so by Rowbottom's theorem there is a measure 1 set $A_p\in F_\lambda\cap M_\lambda$ such that either $[A_p]^{<\omega}\subseteq S_p$ or $[A_p]^k\cap S_p=\varnothing$ for some $k$.
	If there is some $X\in F_\lambda$ such that $\langle p,X\rangle\in D$ then let $X_p$ be the $M_\lambda$-least and set $B_p=A_p\cap X_p$, otherwise let $B_p=A_p$.
	Then let $B$ be the diagonal intersection of the $B_p$'s.
	Being measure 1, $B$ includes a final segment of $\vec c$, so let $q$ be such that $\vec c\setminus q\subseteq B$.
	By density there must be an extension of $\langle q,B\rangle$, say $\langle q\cup t,X\rangle$, lying in $D$.

	If $r\subseteq\vec c$ extends $q$ then $r\setminus q\subseteq B$ by the way we picked $q$.
	By the definition of $B$, $B\smallsetminus q\subseteq A_q$ and so $B\smallsetminus q$ is homogeneous for $S_q$.
	Thus since $t\in S_q\cap B\smallsetminus q$ by definition, $[B\smallsetminus q]^{<\omega}\subseteq S_q$.
	Hence $r\in S_q$, and so for some $Y$ the condition $\langle q\cup r,Y\rangle\in D$.
	Since $q\cup r\subseteq\vec c$, $G\cap D\neq\varnothing$, completing the proof of the first requirement for genericity.

	Now we prove that $G_{\vec c}$ decides every $\Sigma_n$ sentence of $\LF$.
If $\varphi$ is ranked then by Lemma \ref{lem:prikry} for any $p\in[\kappa_\lambda]^{<\omega}$ there is a $\Delta_1^{M_\lambda}(p,\varphi)$-definable set $X$ such that $\langle p,X\rangle\|\varphi$.
	Since $X$ is measure 1 it includes an end-segment of indiscernibles, and since it is definable from $p$ and $\varphi$, it contains every indiscernible beyond $p$.
	Thus if $p\subseteq\vec c$ then $\langle p, X\rangle\in G_{\vec c}$, completing the proof in this case.

	Suppose $\varphi$ is $\exists x\psi(x)$, with $\psi$ a $\Pi_{k-1}$ formula of $\LF$, $k\leq n$, and fix $p\in[\kappa_\lambda]^{<\omega}$.
	Suppose inductively that, for each $t\in \mathcal C$, there is an $X\in F_\lambda$ such that $\langle p, X\rangle\in G_{\vec c}$ and $\langle p, X\rangle\|\psi(t)$.
	Let:
	\begin{align*}
		\theta(X) &\equiv \langle p, X\rangle\Vdash^*\psi(t) \vee \langle p, X\rangle\Vdash^*\neg\psi(t).
	\end{align*}
	This is $\Delta_k$, so for each $t$, let $X_t$ be the $M_\lambda$-least set satisfying $\theta$, hence $\Delta_k^{M_\lambda}$:
	\begin{align*}
		X=X_t &\iff \theta(X) \wedge \forall Y(Y\not<_{M_\lambda} X \vee \neg\theta(Y)) \\
			&\iff \theta(X) \wedge \exists u(u=\operatorname{pr}(X)\wedge\forall Y\in u(\neg\theta(Y))).
	\end{align*}
	Here $\operatorname{pr}$ is the $\Delta_1$ function returning the set of all $<_{M_\lambda}$-predecessors of a set, following \cite{Devlin} where this is defined for $L$ in II.3.5.
	Thus $t\mapsto X_t$ is $\Delta_k^{M_\lambda}$ and the following set is $\Pi_k^{M_\lambda}$:
	\[
		Z = \{\alpha\in\kappa_\lambda\mid\forall t\in \mathcal C (\alpha\in X_t)\}.
	\]
	$Z$ is, by $\Sigma_k$ separation, an element of $M_\lambda$ and $Z\supseteq(\vec c\smallsetminus p)$, so $\mathbf p=\langle p,Z\rangle\in G_{\vec c}$.
	$\mathbf p\|\psi(t)$ for each $t$, so $\mathbf p\|\varphi$, completing the induction and the proof.\end{proof}
\begin{rem}
	Note that we may add finitely many elements of $\kappa_\lambda$ to $\vec c$, and the above proof still works.
	This means that, if $\mathbf p=\langle p, X\rangle\in\P$ then $G_{p\cup\vec c}$ is $M_\lambda$-generic and contains $\mathbf p$.
\end{rem}

\begin{thm}
	$\mathbf p\Vdash\varphi\iff\mathbf p\Vdash^*\varphi$, for $\Sigma_n$ formul\ae{} $\varphi$.
\end{thm}
\begin{proof}
	By $\mathbf p\Vdash\varphi$ we mean the usual (semantic) forcing relation:
	\[
		\forall G(G\text{ is $M_\lambda$-generic for $\P$} \wedge \mathbf p\in G \rightarrow M_\lambda[G]\vDash\varphi).
	\]
	The right-to-left implication of the theorem is provided by the previous lemma.

	For the forward direction suppose $\mathbf p\Vdash\varphi$.
	Suppose for a contradiction that the $\Sigma_n$ class $D=\{\mathbf q\mid \mathbf q\Vdash^* \varphi\}$ is not dense below $\mathbf p$.
	In other words, suppose $\mathbf q\leq\mathbf p$ is such that $\forall\mathbf r\leq\mathbf q(\mathbf r\nVdash^*\varphi)$, i.e. $\mathbf q\Vdash^*\neg\varphi$.
	But then by the right-to-left implication, $\mathbf q\Vdash\neg\varphi$.
	Let $G$ be $M_\lambda$-generic for \P with $\mathbf q\in G$ so $M_\lambda[G]\vDash\neg\varphi$.
	But $\mathbf p\in G$, too, so $M_\lambda[G]\vDash\varphi$, which is a contradiction.
\end{proof}

We have now checked that the usual properties of forcing hold, and can use the $\Vdash$ relation as we normally would, safe in the knowledge that $\mathbf p\Vdash\varphi$ is $\Sigma_n$ as long as $\varphi$ is.

\subsection{The Generic Extension}
\label{sub:The Generic Extension}

We now want to prove that, in the $\Sigma_n$ case of the Truth Lemma, it doesn't matter which initial segment of $\vec c$ we pick; we may always find a suitable measure-1 set $Y$ to force the statement.

\begin{lem}
	\label{lem:forcing-thm}
	Let $p=\{ c_0,\ldots,c_l\}$ and $y$ an arbitrary constant $\Sigma_n$-definable from $\ran\pi$ and finitely many elements of $C$, with the maximum such element $c_j<c_l$.
	Suppose $\psi$ is $\Pi_{n-1}$. 
	Then:
	\[
		M_\lambda[\vec c]\vDash\exists z\psi(z, y) \Leftrightarrow M_\lambda\vDash\exists Y \langle p, 
		Y\rangle\Vdash\exists z\psi(z,y).
	\]
\end{lem}
\begin{proof}
	Fixing $p$, suppose the left hand side holds, and work in $M_\lambda$.
	Then define:
	\[
		A := \{q\in [\kappa_\lambda\smallsetminus p]^{<\omega}\mid\exists z\exists X\langle p\cup q,X\rangle\Vdash\psi(z,y)\}.
	\]
	$A$ is $\Sigma_n^{M_\lambda}(y,p)$, since $\mathbf p\Vdash\psi$ is $\Pi_{n-1}^{M_\lambda}$ if $n\geq 2$ (or $\Delta_1^{M_\lambda}$ if $n=1$.)
	By hypothesis and the truth lemma, there is some extension of $\mathbf p$ in the generic that forces $\exists z\psi(z,y)$.
	Thus let $q\subseteq \vec c\smallsetminus p,X\in F_\lambda,z$ be such that $\mathbf p'=\langle p\cup q, X\rangle\in G_{\vec c}$, $\mathbf p'\Vdash\psi(z,y)$. 
	Hence by indiscernibility, $A$ contains any element of $[C]^{<\omega}$ the same length as $q$, and thus $M_\lambda\vDash A\in (F_\lambda)^k$.
	Hence by Rowbottom in $M_\lambda$, there is $Y\in F_\lambda$ with $[Y]^k\subseteq A$. 
	For such a $q\in A$ let $Y$ be the $M_\lambda$-least such set, so that $Y$ is $\Sigma_n$-definable in $M_\lambda$ from $y$ and $p$.
	Then, if $q\in[Y]^k$ for some $k$, let $z^q, X^q$ be witnesses that $q\in A$.

	The map $q\mapsto X^q$ is therefore $\Sigma_n(y,p)$-definable, allowing us to define:
	\[
		Y' = \Delta_q(X^q\smallsetminus q)\cap Y.
	\]
	$Y'$ is then the diagonal intersection of measure 1 sets and is $\Sigma_n$-definable, so $M_\lambda\vDash Y'\in F_\lambda$.
	$Y'$ also is such that if $q\in[Y']^k$ then $Y'\smallsetminus q\subseteq X^q$, since suppose $\alpha\in Y'\smallsetminus q$ (hence $\alpha > \sup q$), then $\alpha\in X^{q'}\smallsetminus q'$ for all $q'$ with supremum smaller than $\alpha$. But $\sup q < \alpha$ so $\alpha\in X^q$.
	But we have now established that, for any $\langle p\cup q, \bar Y\rangle\leq\langle p,Y'\rangle$, we have $\langle p\cup q, \bar Y\rangle\leq\langle p\cup q,X^q\rangle\Vdash\psi(z^q,y)$ and we are done with the forward direction.

	Hence suppose the latter.
	Then let $\mathbf p = \langle p, Y\rangle$ be the $M_\lambda$-least such condition, hence $\Sigma_n$-definable in $M_\lambda$ from $y, p$.
	But by indiscernibility in the form of Lemma \ref{lem:contains-all-indiscernibles}, $Y\supseteq C\smallsetminus p$, so $\mathbf p\in G_{\vec c}$ by definition of the latter, and we are done.
\end{proof}

As mentioned above, we need to check that the generic extension is a model of \nkp.

\begin{thm}
	\label{thm:forcing-preserves-separation}
	$M_\lambda[\vec c]\vDash\nkp$.
\end{thm}
\begin{proof}
	Since $\kappa_\lambda$ is the largest cardinal in $M_\lambda$, we only need to consider functions defined (partially) on $\kappa_\lambda$.
	Thus for any function $f$ with $\dom f\subseteq\kappa_\lambda$ defined by $f(\xi) = y\Leftrightarrow \exists z\varphi(y,\xi,z,d)$ (where $\varphi\in\Pi_{n-1}^{M_\lambda[\vec c]}$ and $d$ is some parameter, thus $\varphi$ may refer to the generic via the predicate symbol $\dot G$) we wish to prove:
	\begin{align*}
		M_\lambda[\vec c]\vDash \exists\zeta\forall\xi\in\kappa_\lambda(\xi\in\dom f\rightarrow f(\xi) < \zeta).
	\end{align*}
	Thus let $\mathbf p=\langle p, X\rangle=\langle\{ c_0,\ldots,c_k\},X\rangle\in G_{\vec c}$ be a condition which forces the following:
	\[
		f\text{ is a function}\wedge \ran f\subseteq On.
	\]
	The parameter $d$ must, in $M_\lambda$, be definable from $\vec a\in\ran\pi$ and finitely many indiscernibles, $\vec e\in[C]^{<\omega}$ since $M_\lambda$ is an iterated ultrapower model.
	By strengthening $\mathbf p$ if necessary, assume $\max\vec e\leq\max p$. For $\xi\leq\max p$, define $y_\xi$ to be such that:
	\[
		\exists X_\xi (\langle p, X_\xi\rangle\Vdash f(\xi) = y_\xi)
	\]
	wherever such a $y_\xi$ exists. We claim this defines a $\Sigma_n^{M_\lambda}(p,d)$ function defined on $\dom f\cap(\max p+1)$.
	First note that, if there is such a $X_\xi$, then any other $\langle p,X'_\xi\rangle$ is compatible with $\langle p,X_\xi\rangle$ and hence $y_\xi$ is well defined.
	
	For the domain, suppose $M_\lambda[\vec c]\vDash \xi\in\dom f$.
	Then by Lemma \ref{lem:forcing-thm} there is an $X'\in F_\lambda$ such that $\langle p, X'\rangle\Vdash\xi\in\dom f$. Hence $\langle p, X'\cap X\rangle\Vdash\xi\in\dom f$, i.e. $\langle p, X'\cap X\rangle\Vdash\exists y f(\xi) = y$, and so $y_\xi$ must be defined.

	The map $\xi\mapsto y_\xi$ is $\Sigma^{M_\lambda}_n(p,d)$ and so by \nkp has a bound, let's say:
	\[
		\forall\xi\in\dom f(\xi\leq\max p\rightarrow y_\xi <\tau_p).
	\]
	Hence if $\xi\in\dom f,\xi\leq\max p$, we have $\langle p,X_\xi\rangle\Vdash f(\xi)=y_\xi$ and $\langle p, X_\xi\rangle\Vdash f(\xi)<\tau_p$.
	Let $D=\dom f\cap (\max p+1)$, a $\Sigma_n$ element of $M_\lambda$ by $\nkp$.
	Define:
	\begin{align*}
		A=\{r\in[\kappa_\lambda\smallsetminus p]^{<\omega}\mid \exists \tau_{p\cup r},Y_r\forall\xi\in D(\langle p\cup r, Y_r\rangle\Vdash f(\xi)<\tau_{p\cup r})\}.
	\end{align*}
	Now, $A$ is $\Sigma_n$ definable and for each $l$ we know we can find some $r=\langle c_{k+1},\ldots,c_{k+l}\rangle$ such that $r\in A$, since we can find the corresponding $\tau_{p\cup r}$ as we found $\tau_p$.
	Thus by indiscernibility, $A\in (F_\lambda)^l$.
	But then by Rowbottom in $M_\lambda$, there is a set $A_l\in F_\lambda$ such that $[A_l]^l\subseteq A$, and since all parameters in its definition are at most $\max p$, $A_l\supseteq C\smallsetminus p$.
	Then if we let $\tilde A=\bigcap_l A_l$, we have that $\tilde A\supseteq C\smallsetminus p$.
	By \nkp, let $\tau$ bound all the ordinals $\tau_{p\cup r}$ for $r\in[\tilde A]^{<\omega}$.
	
	Hence $\mathbf{\tilde p}:=\langle p, \tilde A\cap X\rangle\in G_{\vec c}$ and for any $\xi\in\dom f$, if $\langle p\cup q, B\rangle$ is an arbitrary extension of $\mathbf{\tilde p}$, we can find some $r\in [B]^{<\omega}$ (and so $q\cup r\in [\tilde A]^{<\omega}$) such that $\xi\leq\max r$ and so $\langle p\cup q\cup r, Y_{q\cup r}\rangle\Vdash f(\xi)<\tau_{p\cup q\cup r}<\tau$.
	Hence $\mathbf{\tilde p}\Vdash\sup\ran f<\tau$ and we are done.
\end{proof}

\begin{thm}
	If $M$ is the $<_*$-least mouse such that $M\vDash\nkp$, then $M_\lambda[\vec c]$ is the smallest transitive model of 
	$\nkp$ containing $\vec c$ as an element. Thus $M_\lambda[\vec c]=\mathcal A_{\nkp}[\vec c]=L_{\wt\theta}[\vec c]$.
\end{thm}
\begin{proof}
	By Separation in $M_\lambda[\vec c]$, we have that $\vec c$ is an element of $M_\lambda[\vec c]$.
	All that remains to prove is minimality.
	Let $N$ be a transitive model of \nkp with $\vec c\in N$.
	We claim that it is sufficient to consider $N$ of the form $J_\alpha^{\vec c}$:
	Otherwise, if $J_\alpha^{\vec c}\not\vDash\nkp$ then there is a $\Sigma_n$-definable partial function definable over $J_\alpha^{\vec c}$ unboundedly into the ordinals of $J_\alpha^{\vec c}$.
	But this function would then be unbounded and $\Sigma_n^{\bar N}$ for any admissible set $\bar N$ with $\vec c\in \bar N$ of height $\alpha$ (since the $J^{\vec c}_\beta$ hierarchy is $\Delta_1$ over any structure containing $\vec c$) including $N$ if $On\cap N=\alpha$.

	Now, we defined $M_\lambda[\vec c]$ to be $J_{\wt\theta}^{\vec c}$.
	Hence we need to show that $J_\alpha^{\vec c}$ is not a model of \nkp for $\alpha<\wt\theta$.

	But by leastness of $M$, $M_\lambda$ is also the least $J_\alpha^{\wt F}$ such that $J_\alpha^{\wt F}\vDash\nkp$.
	Thus if $\alpha<\wt\theta$, $J_\alpha^{\wt F}$ is not a model of $\nkp$ and hence $J_\alpha^{\vec c}$ cannot be a model of \nkp.

	Since $J_{\wt\theta}^{\vec c}$ is admissible, $\wt\theta=\omega\alpha$ and hence in fact $M_\lambda[\vec c]=L_{\wt\theta}[\vec c]$.
\end{proof}

\begin{proof}[Proof of Theorem \ref{thm:nkp-mouse-implies-gamma-n}]
	Since the class $C$ of iteration points is a club, we just have to prove the indiscernibility.

	Here we abbreviate $\mathcal A_{\nkp}[\vec c]$, defined in Definition \ref{dfn:generating-indiscernibles}, by $\mathcal A_n[\vec c]$.
	Since a sentence of the language of $\mathcal A_n[\vec c]$ or $\mathcal A_n[\vec d]$ is a formula of set theory possibly mentioning the predicate symbol $\dot G$, it is sufficient to take such a $\Sigma_n$ sentence $\varphi$ and show that $\mathcal A_n[\vec c]$ and $\mathcal A_n[\vec d]$ agree on $\varphi$.
	Let $M_\lambda$ (respectively $M_{\lambda'}$) be the iterate whose measurable cardinal is the supremum of $\vec c$ (respectively $\vec d$.)
	Then, using that forcing $\varphi$ is $\Sigma_n$:

	\begin{align*}
		\mathcal A_n[\vec d]\vDash\varphi
			&\Leftrightarrow M_{\lambda'}\vDash\exists Y(\langle \varnothing, Y\rangle\Vdash\varphi)
				&&\text{By Lemma \ref{lem:forcing-thm}}\\
			&\Leftrightarrow M_\lambda\vDash\exists Y(\langle \varnothing, Y\rangle\Vdash\varphi)
				&&\text{By $\Sigma_n$-elementarity}\\
			&\Leftrightarrow \mathcal A_n[\vec c]\vDash\varphi
				&&\text{By Lemma \ref{lem:forcing-thm}.}
	\end{align*}
\end{proof}

\section{\texorpdfstring{\ZFC}{ZFC}}

The natural extension of these results to stronger mice will also be interesting.
In this section, $M$ is a non-trivial \ZFC-mouse with normal measure $F$.
We will prove:

\begin{thm}
	\label{thm:zfc-mouse-indiscernibles}
	If there is a non-trivial mouse $M\vDash\ZFC$ then there is a class $C$ of $\Sigma_\omega$ generating indiscernibles for \ZFC.
\end{thm}

Thus these are full indiscernibles and any two models generated by an $\omega$-sequence of them will be fully elementarily equivalent.

The procedure here is much simpler than in the preceding case. Since $M\vDash\ZFC$ again in the language of set theory 
augmented with a predicate for the filter $F$ on $\kappa$, $F$ is a definable set in $M$ by Separation and Power Set. Thus 
the forcing is a set forcing and we can dispense with most of the analogues of the previous section. Define again 
$M_\lambda$ to be the $\lambda$th iterate and have measurable cardinal $\sup\vec c$.

\begin{lem}
	$M_\lambda[\vec c] = \mathcal A_\omega[\vec c]$, the smallest transitive model of \ZFC with $\vec c\in \mathcal A_\omega[\vec c]$.
\end{lem}
\begin{proof}
	Since the forcing is set-sized, this is standard: we know that $M_\lambda[\vec c]\vDash\ZFC$, $\vec c\in 
	M_\lambda[\vec c]$, and that $M_\lambda[\vec c]$ is the smallest transitive model of \ZFC containing $\vec c$.
\end{proof}

\begin{proof}[Proof of Theorem \ref{thm:zfc-mouse-indiscernibles}]
	Let $\varphi$ be any sentence in the language of $\mathcal A_\omega[\vec c]$. Then since we have full elementarity 
	between any two iterates of $M$, 
	\[
		\mathcal A_\omega[\vec c]\vDash\varphi\iff M_\lambda\vDash(\Vdash_\P\varphi)
		\iff M_{\lambda'}\vDash(\Vdash_\P\varphi)
		\iff A_\omega[\vec d]\vDash\varphi
	\]
	using the same notation as in the proof of Theorem \ref{thm:nkp-mouse-implies-gamma-n}.
\end{proof}

\section{Preserving substructures}

Now let us abandon our previous notation and fix $M=\langle J^F_\alpha,F\rangle$ to be an admissible structure of the kind discussed in \cite{Welch96}, that is, the $Q$-structure of a ``clever'' mouse, $M_0$.
Throughout this section, ``admissible'' will mean a transitive model of \KP in the language containing the predicate $F$.
Then $M\vDash\KP\wedge``F\text{ is a normal measure on $\kappa$''}$ and $M$ satisfies the following $\Sigma_1$-Rowbottom property:
\[
	\{\xi<\kappa\mid M\vDash\varphi(\xi,p)\}\in F^{M_0} \implies\exists\tau<\alpha \{\xi<\kappa\mid J^F_\tau\vDash \varphi(\xi,p)\}\in F^M
\]
for $\Sigma_1$ formul\ae{} $\varphi$.
Under these conditions it was proved in \cite{Welch96} that Prikry forcing over $M$ preserves \KP.
Originally this was used under the hypothesis that $M_0$ is the least clever mouse, yielding that $M_\lambda[\vec c]$ (the forcing extension of the $Q$-structure of the $\lambda$th iterate of $M_0$) is the least admissible set containing $\vec c$ as an element.
If $M_0$ is, say, the second clever mouse, then the situation changes as we would hope.

\begin{lem}
	Let $\kappa_\lambda\in N\in M_\lambda$ with $N$ admissible. Let $D\in N$ be predense in $N$, in the partial order $(\P)^N$.
	Then $D$ is predense in $(\P)^{M_\lambda}$. Hence $\vec c$ is $\P$-generic over $N$
\end{lem}
\begin{proof}
	So let $\mathbf p=\langle p,X\rangle\in(\P)^{M_\lambda}$, and we need to find $\mathbf r\leq \mathbf q\in D$ such that $\mathbf r\leq \mathbf p$.
	Now, define the set of extensions of $p$ which extend an element of $D$:
	\[
		H = \{r\in[\kappa\smallsetminus p]^{<\omega}\mid \exists \langle q, Y_r\rangle\in D(p\cup r\supseteq q\wedge p\cup r\setminus q\subseteq Y_r)\}.
	\]
	This is $\Delta_1^{N}(D,p)$, so by admissibility, $H\in N$.
	By Rowbottom, $\exists Z\in N\cap {F_\lambda}$ such that either $[Z]^{<\omega}\subseteq H$ or $[Z]^{<\omega}\cap H=\varnothing$.
	The latter case is impossible since, considering the condition $\langle p,Z\rangle\in (\P)^N$, by predensity there is $r\in [Z]^{<\omega}$ and some $Z'\subseteq Z$ such that $\langle p\cup r, Z'\rangle$ extends an element of $D$.
	That is to say, $\langle p\cup r,Z'\rangle\leq\langle q,Y_r\rangle$ for some condition $\langle q,Y_r\rangle\in D$, i.e. $r\in H$.

	But now $\langle p,Z\cap X\rangle$ is a condition in $M_\lambda$ such that, for any $r\in Z\cap X$, the condition $\mathbf r:=\langle p\cup r,Z\cap X\cap Y_r\rangle\leq\langle p\cup r,Y_r\rangle$ extends a condition in $D$, so $\mathbf p$ is compatible with an element of $D$, which is what we wanted.
\end{proof}

\begin{lem}
	If $N$ in addition is the $Q$-structure of a clever mouse then $M_\lambda[\vec c]\vDash\exists A(A\text{ is 
		admissible})$.
\end{lem}
\begin{proof}
	By the results of \cite{Welch96}, Section 3, $N[\vec c]\in M_\lambda[\vec c]$ is admissible. 
\end{proof}

Now, if $N,M$ are the least and second-least admissible-in-$F$ structures containing $\kappa_\lambda$ as an element, we want to know that $N_\lambda[\vec c], M_\lambda[\vec c]$ are similarly minimal.
For $N_\lambda[\vec c]$ this is already known, and for $M_\lambda[\vec c]$ it is a minor modification:

\begin{lem}
	If $N,M$ are as in the preceding paragraph then $M_\lambda[\vec c]\vDash\text{``$N_\lambda[\vec c]$}$ is the only admissible set $\bar N$ with $\vec c\in\bar N$''.
\end{lem}
\begin{proof}
	Suppose not, i.e. there is a $\gamma\in On\cap M_\lambda[\vec c]$ such that $\gamma>\kappa_\lambda$, $J_\gamma^{\vec c}\vDash\KP$ and $N_\lambda[\vec c]\in J_\gamma^{\vec c}$.
	Then as in the proof of the corresponding Lemma 3.10 of \cite{Welch96}, $J^F_\gamma$ ($F$ the measure on $\kappa_\lambda$) is $Q^{N'}_{\kappa_\lambda}$ for some mouse $N'\in H^{M_\lambda}_{\kappa_\lambda}$.
	$N<_*N'<_*M$, so $N'$ is therefore not clever as there are no clever mice between $N$ and $M$.
	Then obtain a contradiction with the admissibility of $J^{\vec c}_\gamma$ as in \cite{Welch96}.
\end{proof}

\begin{thm}
	\label{thm:two-admissibles-indiscernibles}
	If $M$ is as above, there is a closed-unbounded class of $\Sigma_1$ generating indiscernibles for the theory ``$T=\KP +{}$there is an admissible set $N$ with $\vec c\in N$.''
\end{thm}
\begin{proof}
	Thus with this theory $T$, the model $\mathcal A_T[\vec c]$ is the {second}-least admissible set containing $\vec c$.

	The proof is exactly the same as the above indiscernibility proofs, relying on the fact that forcing relation for $\Sigma_1$ sentences is $\Sigma_1$ over $M_\lambda$, as proved in \cite{Welch96}, and that the iterates are $\Sigma_1$ elementary.
\end{proof}
\let\bibliography\oldbibliography
\let\section\oldsection
\let\subsection\oldsubsection

\section{Determinacy from Indiscernibility}
\label{sec:Determinacy from Indiscernibility}

\let\oldbibliography\bibliography
\def\bibliography#1{\relax}
\let\oldsection\section
\let\oldsubsection\subsection
\let\section\subsection
\let\subsection\subsubsection

Theorems \ref{thm:nkp-mouse-implies-gamma-n}, \ref{thm:zfc-mouse-indiscernibles} and \ref{thm:two-admissibles-indiscernibles} provide the indiscernibles that are used, via the usual argument, to prove determinacy in this section.
From now on, fix a pointclass $\Gamma$, $n\in\omega$ and a theory $T$, and suppose the hypotheses of Theorem \ref{thm:ind-to-det} are satisfied, letting $C$ be the set of generating indiscernibles.

First of all we need a kind of remarkability for generating indiscernibles.
For $\vec c$ a set of ordinals, let $c_m$ denote the $m$th ordinal in the increasing enumeration of $\vec c$.

\begin{lem}
	\label{lem-remarkability}
	Suppose:
	\begin{enumerate}
		\item $\varphi(x)$ is a $\Sigma_n$-formula with one free variable and there is a $m\in\omega$ such that, for any $\vec c\in[C]^\omega$, $\varphi$ defines an ordinal $\gamma$ over $\mathcal A_T[\vec c]$ with $c_{m-1}\leq\gamma<c_m$;
		\item $\vec c,\vec d\in[C]^\omega$ with $\{i\in\omega:c_i\neq d_i\}$ finite and $c_i=d_i$ for $i<m$.
		\item $\mathcal A_T[\vec c]\vDash\varphi(\gamma_1)$ and $\mathcal A_T[\vec d]\vDash\varphi(\gamma_2)$.
	\end{enumerate}
	Then $\gamma_1=\gamma_2$.
\end{lem}

\begin{proof}
	Without loss of generality suppose that $c_m<d_m$.
	Assume further that $\vec c$ and $\vec d$ are cardinals and hence limit points of $C$.
	Doing so does not affect the result; if $\gamma_1\neq\gamma_2$ satisfied the conditions of the Lemma, ``spreading out'' $\vec c$ and $\vec d$ does not change this.
	Let $\vec e=\vec c\cup\vec d$.
	By 2., $e_i=c_i=d_i$ for $i<m$, and $\vec c,\vec d$ are definable in $\mathcal A_T[\vec e]$ (from $\vec e$, which is in our language) by just writing down the points where $\vec c$ and $\vec d$ differ, of which there are finitely many.

	They are hence $\Delta_0$ definable from $\vec e$, and so $\mathcal A_T[\vec c],\mathcal A_T[\vec d]$ are $\Sigma_1$-definable inner models of $\mathcal A_T[\vec e]$.
	Thus the following is expressible as a $\Sigma_n$-sentence of $\mathcal A_T[\vec e]$:
	\[
		\exists \gamma_1<c_m\exists\gamma_2<d_m
		\left(
			\mathcal A_T[\vec c]\vDash\varphi(\gamma_1) \wedge\mathcal A_T[\vec d]\vDash\varphi(\gamma_2) \wedge \gamma_1<\gamma_2
		\right)
	\]
	as is the same sentence but with $\gamma_1>\gamma_2$ instead of $\gamma_1<\gamma_2$.
%

	So, first suppose that $\gamma_1>\gamma_2$, and let $\vec {d'}\in [C]^\omega$ have the same order relationship with $\vec {d}$ as $\vec d$ does with $\vec c$.
	Note this is possible because $\vec c$ and $\vec d$ consist of limit points of $C$ and hence there is enough room to find such a $\vec d'$ from the remaining elements of $C$.
	By indiscernibility we also have $\gamma_2>\gamma_2'$ where $\gamma_2'$ is defined in $\mathcal A_T[\vec e]$ by $\varphi(\gamma_2')^{\mathcal A_T[\vec {d'}]}$.
	Hence by picking successive $\vec d^i$s in the same way, we end up with a decreasing sequence of ordinals --- contradiction.

	Now suppose that $\gamma_1<\gamma_2$.
	There are two further sub-cases to consider, the first being when $\gamma_2<c_j$ for some $j$.
	In this case, we can consider successive $\vec {d'},\vec {d''},\ldots$ where $\vec {d'},\vec d$ have the same order relations as $\vec d,\vec c$ and so on.
	Hence $\gamma_2'$ (defined as before) is greater than $\gamma_2$.
	But we can do this greater than $c_j$-many times as $C$ is a proper class, after which we necessarily have $\gamma_2^{(c_j)}>c_j$.
	But this is impossible by indiscernibility.

	Finally suppose $\gamma_2>c_i$ for all $i$.
	This is still expressible as a $\Sigma_n$ formula inside $\mathcal A_T[\vec e]$ and so this time take successive sequences $\vec {c'}$ so that $\sup c_i'$ approaches $d_i$.
	Thus we can increase $\vec {c'}$ $d_n$-many times, so that $\gamma_2$, which is less than $d_n$, must be less than some $c_i'$, which cannot happen by indiscernibility.

	Thus the only possibility we are left with is the one we wanted: $\gamma_1=\gamma_2$.
\end{proof}

We now commence with the proof of Theorem \ref{thm:ind-to-det}, following the method established by Martin for proving determinacy in the $\ca$ difference hierarchy.
Suppose $A\subseteq\bs$ is $\omega^2\hyp\ca+\Gamma$ and let $\langle B_\beta\mid\beta<\omega^2+1\rangle$ with $B_{\omega^2}\in\Gamma$ witness this. Let $\mathcal A = \mathcal A_T[\langle \aleph_n\mid n<\omega\rangle]$.
We set up an auxiliary game $G^*$ as in, for example, \cite{Martin90}.
For $x\in\bs$, let $\prec_x^\beta$ be a linear ordering of $\omega$ with maximal element 0, such that $\prec_{x\upharpoonright n}^\beta$ depends only on $x\upharpoonright n$ and:
\[ \prec_x^\beta \text{ is a well-ordering if and only if } x\in B_\beta \]
We know such orderings exist by the general theory of $\ca$ and we can in fact make the function $x\mapsto\prec_x^\beta$ recursive.

Next let $\langle \beta,n\rangle\to\xi^\beta_n$ be a bijection between $\omega^2\times\omega$ and $\omega$ such that:
\begin{enumerate}
	\item $\xi^\beta_n$ is even iff $\beta$ is even;
	\item $\xi^\beta_n$ is increasing in $n$ for fixed $\beta$;
	\item For natural numbers $i$ and $a<b$, $\xi^{\omega i+a}_0<\xi^{\omega i+b}_0$
\end{enumerate}
Note this mapping can also be taken to be recursive.

In the auxiliary game $G^*$, each move is a pair $(a_i, \eta_i)$ such that $a_i\in\omega$ and $\eta_i\in\aleph_\omega$:
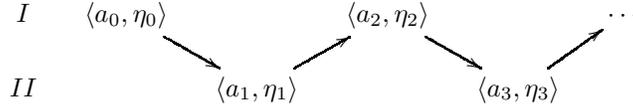
\begin{displaymath}
\xymatrix@C=1pc@R=1pc{
	I & \langle a_0,\eta_0\rangle \ar@/^_/[dr] & & \langle a_2,\eta_2\rangle \ar@/^_/[dr] & & \cdots \\
	II & & \langle a_1,\eta_1\rangle \ar@/^_/[ur] & & \langle a_3,\eta_3\rangle \ar@/^_/[ur]
}
\end{displaymath}
defining an auxiliary game tree, $T^*$.
The tree is definable in a $\Delta^{\mathcal A}_0$ way, and so by admissibility is an element of $\mathcal{A}$.

For $\beta<\omega^2$ and a play of this game, $x^*=\langle(a_i,\eta_i)\mid i\in\omega\rangle$, define the function $F^\beta:\omega\to\aleph_\omega$ by:
\[ F^\beta(n) = \eta_{\xi^\beta_n} \]
Then we say that a play of $G^*$ is \emph{badly lost} if some $F^{\omega\cdot i+b}$ is not an order-preserving embedding of $\langle\omega,\prec_x^{\omega\cdot i+b}\rangle$ into $\langle\aleph_i,<\rangle$.
A badly lost position is defined in the same way.
If $x^*$ is a badly lost play, then there is a shortest badly lost position $p^*\subseteq x^*$ witnessing that some $F^{\omega\cdot i+b}$ is not order preserving.
We then say $x^*$ is badly lost \emph{for $I$} if the least such $b$ is even, otherwise it is badly lost for $II$, and denote the sets of such plays as $B^I$ and $B^{II}$, respectively.

Then, if $x^*$ is badly lost for $I$, $II$ wins, and vice-versa. If $x^*$ is not badly lost then $I$ wins if and only if $x^*\in B_{\omega^2}$.

This implicitly defines the winning set of the game $G^*$, which we denote $A^*$.
We now use the work of Section \ref{sub:Effective Descriptive Set Theory in Uncountable Spaces} to calculate the complexity of $A^*$.

\begin{lem}
	\label{lem:winning-set-complexity}
	$A^*$ is $\wt\Sigma^0_1\wedge\wt\Pi^0_1$ if $n=1$ and $\wt\Sigma^0_n$ if $n>1$ in the space $[T^*]$.
\end{lem}
\begin{proof}
	$A^*$ is defined as:
	\begin{align*}
		x\in A^* &\iff x\notin B^I \wedge (x\in B^{II} \vee \pi(x)\in A_{\omega^2})
	\end{align*}
	where $\pi$ is the projection function from $[T^*]$ to \bs:
	\[
		\pi(\langle\langle a_0,\eta_0\rangle,\langle a_1,\eta_1\rangle,\langle a_2,\eta_2\rangle,\ldots\rangle)
			= \langle a_0, a_1, a_2, \ldots\rangle.
	\]
	Observe that $\pi$ is continuous, and in fact the relation $p^*\subseteq \pi(x^*)$ (for $p^*\in T^*$, $x^*\in[T^*]$) is generalised semi-recursive. Hence, since $A_{\omega^2}$ is $\wt\Sigma^0_n$ considered as a subset of $[T^*]$ by Proposition \ref{prop:relation-old-new-hierarchies}, the assertion ``$\pi(x)\in A_{\omega^2}$'' is $\wt\Sigma^0_n$ by Lemma \ref{lem:closed-under-recursive-functions}.

	Now, $B^I$ and $B^{II}$ are $\wt\Sigma^0_1$: Let $\varphi(p)$ be the formula:
	\begin{align*}
		\exists b,i &\left(F^{\omega\cdot i+b}_p
			\text{ is not an order-preserving map from $\langle\omega,\prec_p^{\omega\cdot i+b}\rangle$ to $\aleph_i$}\right)\\
		\wedge &\text{ the least such $b$ is even.}
	\end{align*}
	This is $\Delta_1^\mathcal H$ and if $\varphi(p)$ with $p\subseteq q$ then $\varphi(q)$. Then $x\in B^I$ if and only if $\exists m(\varphi(x\upharpoonright m))$, so $B^I$ is $\wt\Sigma^0_1$. Similarly for $B^{II}$.
	Hence $A^*$ is $\wt\Pi^0_1\wedge(\wt\Sigma^0_1\vee\wt\Sigma^0_n)$. If $n=1$ this is $\wt\Pi^0_1\wedge\wt\Sigma^0_1$ and if 
	$n>1$ then this is $\wt\Sigma^0_n$.
	
\end{proof}

Now suppose that, by the hypothesis of Theorem \ref{thm:ind-to-det}, $\sigma^*$ is a $\Sigma_n$-definable winning strategy for $G^*$.

We argue that whichever player has the strategy $\sigma^*$ has a winning strategy for $G$.
The strategy for $G$ is defined by having the player ``pretend'' their opponent is playing $G^*$, using indiscernibility to find that the integer moves returned by $\sigma^*$ are independent of the ordinals they choose to pretend were played.
The auxiliary game is constructed so that the integers played constitute a winning play in $G$, so a strategy so defined will also be winning.
We prove this below in the case when $II$ has a winning strategy; the case for $I$ is identical.

When constructing the pretend moves, we will only be considering sequences where $I$ \emph{plays well}, i.e.
he plays so as not to be badly lost and further:
\begin{enumerate}
	\item $F^{\omega i+2j}(n)>F^{\omega i+j'}(0)$ for $j'<2j$, $i,j\in\omega$;
	\item $F^{\omega i+2j}(n)>\aleph_i$ for $i,j\in\omega$;
	\item $F^{\omega i+2j}(n)\in C$ for $i,j\in\omega$.
\end{enumerate}

We now seek to establish the independence of moves from choice of indiscernibles.
Suppose $\langle \eta_0,\ldots,\eta_{2n}\rangle$ and $\langle \eta'_0,\ldots,\eta'_{2n}\rangle$ are two sequences of indiscernibles and
\[
	p=\langle\langle a_0,\eta_0\rangle,\ldots,\langle a_{2n},\eta_{2n}\rangle\rangle,\qquad p'=\langle\langle
 a_0,\eta'_0\rangle,\ldots,\langle a_{2n},\eta'_{2n}\rangle
\]
are two positions consistent with $\sigma^*$ in which $I$ has played well.
We want that the number components of $\sigma^*(p)$ and $\sigma^*(p')$ agree, and further that the ordinal components match, subject to certain conditions.
To be precise:

\begin{lem}
	\label{lem:moves-independent}
	With $p,p'$ as above, if $\sigma^*(p)=\langle a_{2n+1},\eta_{2n+1}\rangle$ and $\sigma^*(p')=\langle a'_{2n+1},\eta'_{2n+1}\rangle$ then:
	\begin{enumerate}
		\item $a_{2n+1}=a'_{2n+1}$
		\item If $2n+1=\xi^\gamma_r$ and whenever $i\leq n$ such that $2i=\xi^\beta_l$, (for $l,r\in\omega$ and $\beta<\gamma$) we have that $\eta_{2i}=\eta'_{2i}$, then $\eta_{2n+1}=\eta'_{2n+1}$.
That is, the ordinal part of the move $\sigma^*$ outputs is not dependent on $I$'s ordinal moves $\eta_{\xi^\beta_r}$ for $\beta\geq\gamma$.
	\end{enumerate}
\end{lem}
\begin{proof}
This is directly analogous to \cite{DuBose90}, Lemmas 0.8.1, 0.8.2.
Let $\vec\kappa=\langle\kappa_m\mid m<\omega\rangle$ be the increasing enumeration of $\{\eta_{2i}\mid i\leq n\}\cup\{\aleph_i\mid i<\omega\}$, with $\vec\kappa'$ defined analogously but with the $\eta'_{2i}$s.
Also let $J$ be the set $\{j<\omega\mid\exists i(\kappa_j=\eta_{2i})\}$, the indices in the enumeration which are ordinals played by $I$.
Since there have only been finitely many moves so far, this set is finite.
Also, $J=\{j<\omega\mid\exists i (\kappa'_j=\eta'_{2i})\}$ since $I$ plays well, so all his ordinals must be played into the appropriate block at the appropriate turn.

Since the ordinal sequences differ only finitely, $\mathcal A_T[\vec\kappa]$ and $\mathcal A_T[\vec\kappa']$ have the same domains as $\mathcal A$.
In all of these structures, the game $G^*$ is definable from the $\aleph_i$s, and the sequence $\langle\aleph_i\mid i<\omega\rangle$ is definable (uniformly in each structure) from $J$, which is finite.
Likewise, the set of non-losing positions for $II$ is $\Pi_1$ over all these structures (we already showed it is $\Pi_1$ over $\mathcal A$), and is the same set in each.

$II$'s integer move $a_{2n+1}$ is just defined by the $\Sigma_n$ formula $a_{2n+1} = (\sigma^*(p))_0$ (that is, the first component of the output of $\sigma^*$).
Further, $a'_{2n+1}$ is an integer with the same $\Sigma_n$-definition in $\mathcal A_T[\vec\kappa']$, which is $\Sigma_n$-elementarily equivalent to $\mathcal A_T[\vec\kappa]$, so $a_{2n+1}=a'_{2n+1}$.

For the second part we are dealing with ordinals, not integers, and therefore use the remarkability established in Lemma \ref{lem-remarkability}.
Fix $m$ to be the least number such that there exists $i\leq n$ and $\beta>\gamma$ with $\kappa_m=\eta_{2i}=\eta_{\xi^\beta_l}$.
Note that if such an $\eta_{2i}$ does not exist then the result is trivial since all $I$'s ordinal moves were the same between $p$ and $p'$.
Now, if $j<m$ then $\kappa_j=\kappa_j'$ (by hypothesis) and:
\begin{equation}
	\kappa_{m-1} \leq\eta_{2n+1}<\kappa_m; \qquad \kappa'_{m-1}\leq\eta'_{2n+1}<\kappa'_m
\label{eq-using-remarkability}
\end{equation}
recalling that $2n+1=\xi^\gamma_r$.
But, together with the $\Sigma_n$-definition of $\eta_{2n+1}$ in $\mathcal A_T[\vec\kappa]$ from $I$'s moves, which are in turn definable from $\{\kappa_j\mid j\in J\}$.
$\eta'_{2n+1}$ has the same $\Sigma_n$-definition in $\mathcal A_T[\vec\kappa']$ from $\{\kappa'_j\mid j\in J\}$ and so we can apply Lemma \ref{lem-remarkability} and $\eta_{2n+1}=\eta'_{2n+1}$.
\end{proof}

So we define $II$'s strategy in $G$ as usual: Suppose we have defined $\sigma(p)$ for $p$ of length at most $2n$.
Then let $p=\langle a_0,\ldots,a_{2n}\rangle$ be a position in $G$ compatible with $\sigma$ as defined so far, and suppose we can find $p^*=\langle\langle a_0,\eta_0\rangle,\ldots,\langle a_{2n},\eta_{2n}\rangle\rangle$, a position in $G^*$ consistent with $\sigma^*$ where $I$ has played well.
Then, set $\sigma(p) = (\sigma^*(p))_0$.
That this is well-defined will be shown in the following lemma:

\begin{lem}
	\label{lem-II-winning}
	If $\sigma^*$ is a winning strategy for $II$ in $G^*$ from the perspective of $\mathcal A$, then $\sigma$ is, in $V$, winning for $II$ in $G$.
\end{lem}
\begin{proof}
	Let $x=\langle a_0,a_1,\ldots\rangle$ be a play in $G$ consistent with $\sigma$.
	If we can show that for any even $\gamma<\omega^2$, $x\in\bigcap_{\beta\leq\gamma}B_\beta\rightarrow x\in B_{\gamma+1}$, and that if $x\in\bigcap_{\beta<\omega^2}B_\beta$ then $x\in B_{\omega^2}$, then we will have shown that $x\notin A$, which is what we want.

	We do this by constructing the ordinals $F^\beta(i)$ in such a way that each $F^\beta$ is order-preserving.
	Inductively assume we have defined $F^{\beta'}(i)$ for all $i$ and all $\beta'<\beta$, for some $\beta\leq\gamma+1$, and we will define the ordinals $F^\beta(i)$.

	Start by picking $I$'s ordinals, i.e. when $\beta$ is even:
	For each $i<\omega$ and $\xi=\sup_{\beta'<\beta} F^{\beta'}(0)$, let $F^{\beta}:\omega\to\aleph_{k+1}$ (where $\omega\cdot k\leq\beta<\omega\cdot{k+1}$) embed $\prec_x^\beta$ order-preservingly into $C\cap\aleph_{k+1}\setminus\xi$, i.e. let $I$ play well.
	We know this is possible because $x\in B_\beta$ and $C$ is closed and unbounded below $\aleph_{k+1}$.

	To pick $II$'s ordinals, i.e. when $\beta$ is odd, we use $\sigma^*$ inductively.
	Let $p^*$ be the position of length $\xi^\beta_m$ whose number components come from $x$ and whose ordinal components are picked as follows:  
	$\eta_i$ is already defined by induction for $i=\xi^{\beta'}_{m'}$ if $\beta'<\beta$.
	Otherwise if $i$ is odd, $\sigma^*$ specifies $\eta_i$, whilst if it is even we let $\eta_i$ be an arbitrary element of $C$, maintaining that $I$ plays well.
	This latter condition can always be met since $C$ is closed and unbounded below each uncountable cardinal.
	Now, by Lemma \ref{lem:moves-independent}, $\sigma^*(p^*)$ is the same regardless of the arbitrary indiscernibles we choose for $I$, so is well-defined, and we define $\eta_{\xi^\beta_m}$ as the ordinal component of $\sigma^*(p^*)$.

	We have now defined all ordinals $\eta_i^{\beta}$ for $\beta\leq\gamma+1$.
	For $\beta=\gamma+1$, since $\sigma^*$ is winning in $\mathcal A$, this tells us that the function $F^{\gamma+1}:\prec^{\gamma+1}_{x\upharpoonright k}\to \aleph_\omega$ given by $i\mapsto\eta^{\gamma+1}_i$ is order-preserving for each $k$; otherwise the shortest $x\upharpoonright k$ where this was not possible would witness that $II$ lost badly while following $\sigma^*$.
   	Since $\sigma^*$ is winning, this cannot happen.
   	Hence $\prec^{\gamma+1}_x = \bigcup_{i\in\omega}\prec^{\gamma+1}_{x\upharpoonright i}$ can be mapped order-preservingly into the ordinals, and $x\in B_{\gamma+1}$ as desired.

	All that is left is then to show that if $x$ is in all the $B_\beta$s up to $\omega^2$, then $x\in B_{\omega^2}$.
	This is true in $\mathcal A$ since $\sigma$ is defined from $II$'s winning strategy $\sigma^*$, so we just need to know that it is true in $V$.
	But the sentence ``there exists a real $x$ consistent with $\sigma$ in all $B_\beta$s except $B_{\omega^2}$'' is $\Sigma^1_2(\sigma)$, and thus by \ref{thm:generalised-shoenfield} is absolute for $\mathcal A$.
	Hence if $\sigma$ were not winning in $V$, it would not be winning in $\mathcal A$, which is a contradiction.

	Hence $x\notin A$, and $\sigma$ is winning.
\end{proof}

The corresponding proof for $I$ is all but identical, and thus we have proved Theorem \ref{thm:ind-to-det}.
\let\bibliography\oldbibliography
\let\section\oldsection
\let\subsection\oldsubsection

\section{Individual Determinacy Proofs}
\label{sec:Individual Determinacy Proofs}

To complete the proof of the main theorem, we need to show that determinacy of the class $\wt\Gamma$ holds in the models $\mathcal A_T[\vec c]$, for the relevant theories $T$.
The class $\wt\Gamma$ is precisely the class given by Lemma \ref{lem:winning-set-complexity}, and we now prove the determinacy of each of these classes in the corresponding model, which are the forcing extensions found in Section \ref{sec:Obtaining Indiscernibles}.

\subsection{\texorpdfstring{$\Det(\wt\Sigma^0_1\wedge\wt\Pi^0_1)$}{Det tilde-Sigma-0-1 and tilde-Pi-0-1}}
\label{sub:Sigma-0-1}

\let\oldbibliography\bibliography
\def\bibliography#1{\relax}
\let\oldsection\section
\let\oldsubsection\subsection
\def\section#1{\let\section\subsubsection}
\let\subsection\subsubsection
\begin{lem}
	Let $N$ be admissible with $T^*\in N$, and $\tau$ be the canonical winning strategy for $II$ in an open game $G(A;T^*)$.
	Then
	\begin{enumerate}
		\item $\tau$ is $\Sigma_1\wedge\Pi_1$-definable over $N$; and
		\item is still winning in any admissible $M\ni N$.
	\end{enumerate}
\end{lem}
\begin{proof}
	Recall that, if $II$ has a winning strategy in an open game, the canonical one is given by never playing so as to end up in a \emph{ranked} position, where the rank of a position $p$ is 0 if it is already lost for $II$ (i.e. already in one of the basic open sets making up the winning set) and otherwise:
	\[
		\rk(p) = \mu.\xi(\exists a\forall b(p\concat\langle a,b\rangle \in T^* \wedge \rk(p\concat\langle a,b\rangle)<\xi)).
	\]
	By the recursion theorem, $\rk$ is a $\Sigma_1^\KP$ function.
	$\tau$ is then defined by 
	\[
		\tau(p)=a\leftrightarrow a\text{ is not ranked}\wedge\forall a'<_N a(a\text{ is ranked}).
	\]
	``$a$ is ranked'' is $\Sigma_1$ (it is equivalent to $\exists\xi\rk(a)=\xi$) so this is $\Pi_1\wedge\Sigma_1$, as desired for the first claim.

	For the second claim, let $N\in M$ with the latter admissible.
	If $\tau$ is not winning in $M$ then some position $p$ consistent with $\tau$ must have a game rank in $M$ but not in $N$, so suppose $p$ is such a position, with rank least.
	Then by definition,
	\[
		M\vDash\exists a\forall b(\rk(p\concat\langle a,b\rangle) < \rk(p)).
	\]
	But, fixing $a$, any such $p\concat\langle a,b\rangle$ must thus also have a rank in $N$, since $p$ is the rank-minimal position such that $p$ does not have a rank in $N$.
	Hence by admissibility $p$ can be ranked in $N$ as $\sup_{b}\rk(p\concat\langle a,b\rangle)$, which is a contradiction.
\end{proof}

The following is essentially from \cite{Tanaka1990} but translated into a set-theoretical form. The original proof uses 
$\Pi^1_1$ comprehension in the $Z_2$ context, and our hypothesis essentially gives us $\wt\Pi^1_1$ comprehension for the 
second admissible above the game-tree.
\begin{thm}
	Let $M$ be an admissible set with $T^*\in M$ such that
	\[
		M\vDash\exists M_0(\Trans(M_0)\wedge T^*\in M_0\wedge(\mathsf{KP})^{M_0}).
	\]
	Then any $\wt\Sigma^0_1\wedge\wt\Pi^0_1$ game has a winning strategy definable over $M$ by a boolean combination of $\Sigma_1$ formul\ae.
	Hence a further admissible set containing $M$ will contain the winning strategy as an element.
\end{thm}
\begin{proof}
	Let $A$ be $\wt\Sigma^0_1$ and $B$ be $\wt\Pi^0_1$, and fix the game $G$ where player $I$ is trying to get into $A\cap 
	B$. Now a winning strategy for the game with winning set $A$ (and likewise for $B$) is definable over the admissible 
	$M_0$ via a game-rank argument, but depending on who wins, is not necessarily a member.

	For $p\in T^*$, $\plays{T^*_p}$ is the open neighbourhood of $p$. Suppose player $I$ has no winning strategy to play 
	into $A\cap B$, and define the set $z\subseteq T^*$:
	\[
		z = \{p\in T^* \mid |p| \text{ is even} \wedge \plays{T^*_p}\subseteq A \wedge \exists\sigma (\text{$\sigma$ is a 
			w-s for $I$ into $B\cap \plays{T^*_p}$})\}.
	\]
	$z$ is thus a $\wt\Sigma^1_1$ generalised real and an element of $M$ by the generalised Spector-Gandy theorem. The class of 
	extensions of elements of $z$ is $\wt\Sigma^0_1(z)$ and thus has a winning strategy $\Pi_1\wedge\Sigma_1$-definable over 
	$M$ since $z\in M$.  In fact the strategy must be for $II$, since a winning play for $I$ would also be in $A\cap B$, and 
	$I$ has no winning strategy there, by assumption. Call $II$'s strategy $\tau_0$.

	Define
	\[
		y = \{ p\in T^*\mid |p|\text{ is even} \wedge \plays{T^*_p}\subseteq A\wedge p\notin z\}.
	\]
	So if $p\in y$ then by open determinacy, $II$ has a winning strategy to play out of $B\cap\plays{T^*_p}$. Let $\tau_p$ 
	be this strategy for each such $p$. Such strategies are definable over $M_0$, thus $\Delta_1$-definable elements of $M$ 
	and so $p\mapsto\tau_p$ is a $\Sigma_1^M$ function by admissibility.

	Then we define a winning strategy $\tau^*$ for $II$ by setting $\tau^*(p) = \tau_0(p)$ whenever 
	$\plays{T^*_p}\not\subseteq A$, and $\tau^*(p\concat r) = \tau_p(r)$ when $p$ is the shortest position such that 
	$\plays{T^*_p}\subseteq A$. Thus if $\tau_0$ leads $II$ to a position $p$ which will end up in $A$, she switches to 
	$\tau_p$ which will keep her out of $B$ forever, thus winning. Since $\tau_0$ and $\tau_p$ are definable over $M$, so is 
	$\tau$.
	\[
		\tau(p) =
		\begin{cases}
			\tau_0(p) &\text{if } \plays{T^*_p}\not\subseteq A;\\
			\tau_q(p) &\text{otherwise, where $q\subseteq p$ is shortest such that } \plays{T^*_q}\subseteq A.
		\end{cases}
	\]
	This is definable over $M$ by a boolean combination of $\Sigma_1$ formul\ae.
	If $N$ is an admissible set with $M\in N$ then $\tau_0,\tau_p$ will be in $N$, and still be winning for their respective subgames.
	Hence $\tau\in N$ and is also winning.
\end{proof}
\let\bibliography\oldbibliography
\let\section\oldsection
\let\subsection\oldsubsection

Thus, since starting with the second-least clever mouse yields generating indiscernibles for such structures $M$, Theorem \ref{thm:ind-to-det} applies, and we have proved Theorem \ref{thm:main} part 1.

\subsection{\texorpdfstring{$\Det(\wt\Sigma^0_2)$}{Det tilde-Sigma-0-2}}
\label{sub:Sigma^0_2}

\let\oldbibliography\bibliography
\def\bibliography#1{\relax}
\let\oldsection\section
\let\oldsubsection\subsection
\def\section#1{\let\section\subsubsection}
\let\subsection\subsubsection
\section{Wolfe's proof}

We show that if $A$ is $\wt\Sigma^0_2$ then $G(A;T^*)$ is determined in transitive models of \skp.
Let $T$ be \skp, then let $\mathcal A_T[\vec c]$ be the least transitive model of $T$ containing $\vec c$ as an element, where we assume $\{\aleph_i\mid i\in\omega\}\subseteq\vec c$.
Being a model of \skp{} implies that there are admissibles $N_i$ such that:
\[
	N_1\prec_{\Sigma_1}N_2\prec_{\Sigma_1}\ldots\prec_{\Sigma_1} \mathcal A_T[\vec c].
\]
Being a minimal transitive model implies that $\mathcal A_T[\vec c]$ has a parameterless $\Sigma_1$ Skolem function.

The determinacy proof is just the same as usual, first proved by Wolfe in \cite{Wolfe55}, and also found in \cite{MartinUP} as Theorem 1.3.3., but we have to pay attention to how much separation and replacement we're using.
For convenience, fix $\mathcal A$ to be $\mathcal A_T[\vec c]$ for some arbitrary $\vec c$ (for example, the sequence of $\aleph_i$s).

We first need to know the complexities of some standard concepts. Most importantly, ``being a winning strategy'' is $\Delta_1$ in a limit of admissibles:

\begin{lem}
	\label{lem:winning-is-delta-1}
	The sentence ``$\sigma$ is a winning strategy for $G(A;T^*)$'' (for $I$, $II$ or either player) is $\Delta_1$ over any admissible $M$ satisfying:
	\begin{enumerate}
		\item $T^*\in M$;
		\item $\exists N\in M(N\text{ is admissible and }\sigma\in N)$
	\end{enumerate}
	Hence the sentence ``there exists a winning strategy for $G(A;T^*)$'' (for $I$, $II$ or either player) is $\Sigma_1$ over $\mathcal A$.
\end{lem}
\begin{proof}
	Let $\sigma\in N\in M$, with $N, M$ admissible. Then certainly being winning is expressible by the $\Pi_1$ formula:
	\[
		\forall x(x\in\plays{T^*}\rightarrow\sigma * x \in A).
	\]
	Note that checking $x\in A$ for a generalised-arithmetic $A$ is $\Delta_1$ over $M$ since it is definable over $\mathcal H$.

	Now we just need a $\Sigma_1$ formula. But note that the set $\{x\mid x*\sigma\in A\}$ is $\wt\Sigma^1_1$ and so by Theorem \ref{thm:generalised-kbt}, if non-empty, has an element definable over $N$, and hence in $M$, witnessing that $\sigma$ is not winning. Thus if $M$ is as required, the following sentence will do:
	\[
		\exists N(\Trans(N) \wedge N\vDash\KP \wedge T^*\in N \wedge \forall x\in\plays{T^*}\cap \operatorname{Def}(N)(\sigma * x\in A)).
	\]
	Again, everything inside the scope of the quantifier is $\Delta_1$, so we are done with the first assertion of the lemma. The second assertion then follows trivially.
\end{proof}

Now let's examine quasi-strategies:

\begin{lem}
	If $G(A;T^*)$ is not a win for $I$, then $II$'s non-losing quasi-strategy is a $\Pi_1$ definable element of $\mathcal A$.
\end{lem}
\begin{proof}
	``not being a win for $I$'' is a $\Pi^{\mathcal A}_1$ property by the previous lemma, and in this case, $II$'s non-losing quasi-strategy is $\Pi_1^{\mathcal A}$:
	\[
		T' := \{p\in T^*\mid \forall n\leq |p|(G(A;T^*_{p\upharpoonright n}) \text{ is not a win for $I$})\}.
	\]
	Being $\Pi_1$, $T'$ is $\Delta_0$ definable from its $\Sigma_1$ complement in $T^*$.
	Thus by $\Sigma_1$-Separation, $T'$ is an element of $\mathcal A$.
\end{proof}

We now proceed with Wolfe's proof, starting with the following Lemma:

\begin{lem}
	\label{lem-positions-for-closed-subset}
	Let $B\subseteq A$ with $B\in\wt\Pi^0_1$ and $A$ a generalised-arithmetic class of $\mathcal A$.
	If $I$ has no winning strategy for $G(A;T^*)$, then in $\mathcal{A}$ there is a strategy $\tau$ for $II$ such that for any play $x\in\plays\tau$, $x$ extends a position $p$ such that:
	\begin{enumerate}
		\item $\plays {T_p^*} \cap B=\varnothing$;
		\item $G(A;T_p^*)$ has no winning strategy for $I$.
	\end{enumerate}
\end{lem}

\begin{proof}
	Define $C$ to be the set of plays such that no initial move satisfies both the assertions of the lemma, that is:
	\[
		C:=\{x\in\plays T^*\mid \forall p\subseteq x (\plays{T^*_p}\cap B\neq\varnothing \vee G(A;T^*_p) \text{ is a win for $I$})\}.
	\]
	Being a win for $I$ is $\Sigma_1$ and the rest of the definition here is $\Delta_0$, so $C$ is $\Sigma_1$-definable over $\mathcal{A}$.
	To prove the lemma we just need a winning strategy for $II$ in the game $G(C;T^*)$, so suppose there isn't one.
	But $C$ is closed; it is the complement of the union of basic open sets and so $I$ must have a winning strategy, $\sigma$.
	Now, since $G(A;T^*)$ is not a win for $I$, by the previous lemma $II$'s non-losing quasi-strategy $T'$ is a $\Pi_1^{\mathcal A}$ element of $\mathcal A$.
	Hence $I$ has no winning strategy in $G(A;T')$.
	But since $T'$ doesn't restrict $I$'s moves, $\sigma$ restricted to $T'$ is still a winning strategy for $G(C;T')$.

	Denote $\sigma$'s restriction to $T'$ by $\sigma'$, and let $x$ be any play consistent with $\sigma'$.
	Then for every $p\in T'$, and hence for any $p\subseteq x$, the game $G(A;T_p^*)$ is not a win for $I$ since $T'$ is $II$'s non-losing quasi-strategy.
	Hence we have shown that the second condition above holds for every $p\subseteq x$.
	Since we assumed that not both are true, then, every $p\subseteq x$ satisfies $\plays{T^*_p} \cap B\neq\varnothing$.
	But $B$ is closed so this means it contains $x$.
	Hence $x\in B\subseteq A$, but $x$ was an arbitrary play consistent with $\sigma'$, so $\sigma'$ is winning for $G(A;T')$, which is a contradiction.
\end{proof}

The proof of the theorem now follows quite simply.

\begin{thm}
	$\mathcal{A}\vDash\Det(\wt\Sigma^0_2)$.
\end{thm}
\begin{proof}
	Let $A\in\wt\Sigma^0_2$, so that $A=\bigcup_{i\in\omega} A_i$ for closed $A_i$.
	Suppose $G(A;T^*)$ is not a win for $I$, and we will construct a winning strategy $\tau$ for $II$.
	The familiar idea is to build $\tau$ from countably many $\tau_i$, each given by applying the lemma with $B$ set to $A_i$.

	Let $\tau_0$ be the strategy given by the previous lemma for $B=A_0$, and let $p_0$ be the shortest position consistent with $\tau_0$ and $q^I$ satisfying both conditions stated in the lemma.
	We then simultaneously define $p_i$ and $\tau_i$ by induction, for as long as $p_i$ is shorter than $2n$.

	Firstly:
	\begin{align*}
		S(T_p^*) &= \parbox{7cm}{the least strategy $\tau$ that wins $G(C;T_p^*)$, with $C$ as defined in the previous lemma.}\\
		P(\tau_i,q^I) &= \parbox{7cm}{the shortest position $p\subseteq\tau_{i}*q^I$ satisfying both conditions of the previous lemma.}
	\end{align*}
	(Here we interpret $\tau_i*q^I$ for the finite sequence $q^I$ as being all plays compatible with $\tau_i$ where $I$ played from $q^I$ his first $n$ moves.)
	Then, if we have defined all strategies and positions up to $\tau_i,p_i$:
	\begin{align*}
		\tau_{i+1}(p) &= \begin{cases}
			\tau_{i}(p) &\text{if } p\subseteq p_i\\
			S(T_{p_i}^*) &\text{otherwise}
		\end{cases}\\
		p_{i+1} &= P(\tau_{i+1},q^I).
	\end{align*}

	Note that $p_{i+1}$ (and hence $\tau_{i+2}$) will be undefined if the lemma would produce a position longer than $2n$.

	Let $\tau(\varnothing)=\tau_0(\varnothing)$.
We define a winning strategy for $II$ by recursion on length of position.

	Suppose we have defined $\tau$ on all positions of length up to $2n$, and let $q$ be a position of length $2n+1$ compatible with $\tau$.
	Let $q^I=\langle q_0,q_2,\ldots,q_{2n}\rangle$, and use it to define $\tau_i$ and $p_i$ as above, with $k$ least such that $p_k$ is undefined.
	Then we define $\tau(q)=\tau_k(q)$.
	Note that $S$ and $P$ are $\Sigma_1$ and $\Delta_0$ over $\mathcal A$, respectively, so this amounts to a $\Sigma_1$ recursion, giving $\tau$ a $\Sigma_1$-definable element of $\mathcal A$, so it just remains to show that $\tau$ is winning.

	So let $x$ be consistent with $\tau$, so that we have positions $p_i\subseteq x$ for all $i\in\omega$, with $\plays{T^*_{p_i}}\cap A_i=\varnothing$.
	Hence $x\in\plays{T^*_{p_i}}$ for each $i$, and so $x\notin\bigcup_{i\in\omega}A_i=A$, so $II$ wins, and $\tau$ is winning.
\end{proof}

This gives us that, if $I$ does not have a winning strategy in $\mathcal A$, $II$ does, and in fact it is $\Sigma_1$ definable.
If $I$ does have the winning strategy, we can in fact find a $\Sigma_1$ definable one because, by Lemma \ref{lem:winning-is-delta-1}, ``being a winning strategy'' is $\Delta_1$ definable over $\mathcal A$ and thus the $\mathcal A$-least such strategy is $\Sigma_1^{\mathcal A}$.
\let\bibliography\oldbibliography
\let\section\oldsection
\let\subsection\oldsubsection

Thus we have proved Theorem \ref{thm:main} part 2.

\subsection{\texorpdfstring{$\Det(\wt\Sigma^0_3)$}{Det tilde-Sigma-0-3}}
\label{sub:Sigma^0_3}

\let\oldbibliography\bibliography
\def\bibliography#1{\relax}
\let\oldsection\section
\let\oldsubsection\subsection
\def\section#1{\let\section\subsubsection}
\let\subsection\subsubsection

\section{The Stuff}

Let $\mathcal T$ be the theory $\KP+\Sigma_2\hyp\mathsf{Sep}$, and $\mathcal{A}_\mathcal T[\vec c]$ the least transitive model of $\mathcal T$ containing $\vec c$ as an element.
For simplicity of notation let $\mathcal A$ be $\mathcal A_\mathcal T[\vec c]$ for some $\vec c$ including $\{\aleph_i\mid i\in\omega\}$.
The game tree $T^*$ is a $\Delta_0$ element of $\mathcal A$.
It is a standard fact that models of $\KP+\Sigma_2\hyp\mathsf{Sep}$ are limits of admissibles and more, so being a winning strategy is a $\Delta_1^\mathcal A$ predicate, and the non-losing quasi-strategy is $\Pi_1^\mathcal A$.

Let $A$ be a $\wt\Sigma^0_3$ subset of $\plays{T^*}$, and we will want to show that $\mathcal A\vDash ``G(A;T^*)$ is determined.''
We follow the specialised version of Davis' original proof as laid out in \cite{Welch11}.
This in turn uses Martin's version, as in \cite{MartinUP}, of the proof closely, but explicitly minimising the required strength.
We note that $\Sigma_2\hyp\mathsf{Sep}$ is more than is necessary to effect this proof, so our proof is simpler than that in \cite{Welch11}.

As with Martin's proof, we start with a lemma that is applied repeatedly to build up a strategy to prove the theorem.
We will have to ensure that the objects defined are elements of the structure $\mathcal A$, and must allow for the added complexity coming from the use of the generalised pointclass $\wt\Sigma^0_3$.

\begin{lem}
	Let $B\subseteq A$ be $\wt\Pi^0_2$ and $T$ a game subtree of $T^*$. Suppose $I$ has no winning strategy in $G(A;T)$, then $II$ has a quasi-strategy $\bar T$ such that:
	\begin{enumerate}
		\item $\plays {\bar T}\cap B=\varnothing$;
		\item $G(A;\bar T)$ is not a win for $I$.
	\end{enumerate}
\end{lem}

\begin{proof}
	Let $T'$ be $II$'s non-losing quasi-strategy, which is a $\Pi_1$-definable set in $\mathcal A$.
	Define a position $p$ to be \emph{good} if there is a quasi-strategy $\bar T\subseteq T'_p$ satisfying the above two properties.
	Goodness is a $\Sigma_2$ property, asserting the existence of a tree with properties 1. and 2., with 1.  being $\Delta_1$ and 2. being $\Pi_1$.
	The set of good positions $H$ is therefore an element of $\mathcal A$.

	Denote by $\hat T$ the function with domain $H$, defined over $\mathcal A$ as: $\hat T(p)$ is the least quasi-strategy witnessing that $p$ is good.
	This is $\Sigma_2$ definable by the existence of $\Sigma_2$ Skolem functions.

	Let $B=\bigcap_n D_n$, with each $D_n\in\wt\Sigma^0_1$, and define the sets:
	\[
		E_n= A\cup\{x\in\plays{T'}\mid \exists p\subseteq x(\plays{T'_p}\subseteq D_n\wedge p\text{ is not good})\}.
	\]
	Thus each $E_n$ is a $\Pi_2$ set in $\mathcal A$ once we make the $\exists p$ bounded in the usual way.

	The goal, then, is to prove that the initial position $\varnothing$ is good.
	This will be accomplished by showing that there is at least one $E_n$ for which $I$ has no winning strategy in $T'$.
	Following Martin, we first show that this does what we want, so assume that there is no winning strategy for $I$ in the game $G(E_n;T')$ in $\mathcal A$.
	Let $T''$ be $II$'s non-losing quasi-strategy in this game.

	We now define a quasi-strategy $\bar T$.
	$\bar T$ firstly contains all $p\in T''$ until $\plays{T'_p} \not\subseteq D_n$.
	Then letting $p$ be a minimal position such that $\plays{T'_p}\subseteq D_n$, note that $p$ must be good.
	This is because, if $p$ were bad the definition of $E_n$ implies $I$ has a trivial winning strategy in $G(E_n;T'_p)$.
	But $p$ is supposed to be in $II$'s non-losing quasi-strategy, and hence upon reaching such a $p$ we can include $\hat T(p)$ into $\bar T$.

	Now we wish to show that $\bar T$ is a witness in $\mathcal A$ to the initial position being good.
	Note firstly that $\plays{\bar T}\cap B=\varnothing$, because either a play through $\bar T$ remains in $T''$ (and so it is not in $D_n$, hence not in $B$) or it goes through $\hat T(p)$ which witnesses $p$'s goodness, hence avoiding $B$.

	We now need to show that $G(A;\bar T)$ is not a win for $I$, so suppose otherwise and let $\sigma\in \mathcal A$ witness that.
	Every position $p$ consistent with $\sigma$ satisfies $\plays{T'_p}\not\subseteq D_n$: if not, then there is such a $p$ with $\bar T_p=\hat T(p)$.
	But $\hat T(p)$ is a witness to $p$'s goodness in $\mathcal A$, so $G(A;\bar T_p)$ is not a win for $I$ in $\mathcal A$.

	By the definition of $\bar T$, then, all plays consistent with $\sigma$ lie in $\plays{T''}$.
	Hence $\sigma$ is also a winning strategy for $G(A;T'')$ and, since $A\subseteq E_n$, for $G(E_n;T'')$.
	But this allows us the contradiction we want: Define a strategy $\tau$ for $I$ in $G(E_n;T')$ by following $\sigma$ until possibly reaching a point $p\notin T''$.
	Then there is a strategy $\sigma_p$ winning for $I$ in $\mathcal A$ (since $T'$ is non-losing in $\mathcal A$), so let $\tau$ follow this strategy.
	This procedure is definable from $\sigma$ and the strategies $\sigma_p$, all of which are elements of $\mathcal A$, and thus the strategy is in $\mathcal A$ for the desired contradiction.

	We have now shown that, under the assumption that at least one game $G(E_n;T')$ is not a win for $I$, we have that $\varnothing$ is good as desired, and we need to show that this assumption is valid.
	Note that the above argument can be adapted to show that, if $G(E_n^p;T'_p)$ is not a win for $I$, where:
	\[
		E_n^p := A\cup\{x\in\plays{T'_p}\mid\exists q\subseteq x(p\subseteq q\wedge\plays{T'_q}\subseteq D_n\wedge\text{$q$ is not good})\}
	\]
	then $p$ is good.
	Hence suppose the lemma is false.
	We can thus assume that each $G(E_n^p;T'_p)$ is a win for $I$ (in $\mathcal A$) since this is the only case in which we haven't proved the lemma.
	We then show that $I$ has a winning strategy in $G(A;T')$ in $\mathcal A$, contradicting the hypothesis of the lemma.

	$I$ starts by playing with the strategy from the previous lemma for $G(E_{0};T')$.
	If we never reach a point $p_1$ with $\plays{T'_{p_1}}\subseteq D_0$ then, if the play ends up in $E_{0}$ it must be in $A$.
	Otherwise all plays above $p_1$ are in $D_0$, and $I$ now must start playing according to the strategy for $G(E_1^{p_1};T'_{p_1})$.
	We now either end up in $A$ or use the strategy for $G(E_2^{p_2};T'_{p_2})$.
	If this process terminates, then $I$ has a winning strategy for some $G(E_i^{p_i};T'_{p_i})$ and the only way for $I$ to win in the game would be to get into $A$, giving a winning strategy in $G(A;T')$.
	If the process continues indefinitely, then $I$ has landed in each $D_n$ and hence is in $B\subseteq A$.

	It remains to show that the strategy thus defined is an element of $\mathcal A$, so let's make our definitions above precise.
	Suppose $q$ is an even-length position with $q\supseteq p_i$, and that we have defined $p_i$ and $\sigma_i$, but not $p_{i+1}$ and $\sigma_{i+1}$.
	\begin{align*}
		\sigma(q)	&=
		\begin{cases}
			\sigma_i(q) & \text{if $\plays{T'_q}\not\subseteq D_i$}\\
			\sigma_{i+1}(q) & \text{otherwise}
		\end{cases}\\
		\intertext{where, if ``otherwise,'' we set:}
		\sigma_{i+1}	&= \text{$I$'s least winning strategy in $G(E_i^q;T'_q)$}.
	\end{align*}
	Thus $\sigma$ is defined by a recursion on $\omega$ and is an element of $\mathcal A$.
	But $T'$ is $I$'s non-losing quasi-strategy in $\mathcal A$ and so this is a contradiction.

\end{proof}

\begin{thm}
	For any $\wt\Sigma^0_3$ subset $A$ of $\plays{T^*}$, $\mathcal A\vDash\text{``$G(A;T^*)$ is determined''}$.
\end{thm}

\begin{proof}
	Now we are ready to prove the theorem.
	Let $A=\bigcup_{n\in\omega} A_n$, with each $A_n$ $\wt\Pi^0_2$, and suppose $G(A;T^*)$ is not a win for $I$, and we show that $II$ has a winning strategy definable by a $\Sigma_2$-recursion over $\mathcal A$ which, by $\Sigma_2$-admissibility is therefore an element of $\mathcal A$.
	We apply the previous lemma repeatedly, substituting each $A_n$ as an instance of $B$, and the resulting quasi-strategies as instances of $T$.

	First apply the Lemma with $B=A_0$ and $T^*$, yielding a quasi-strategy $\bar T^*\in \mathcal A$, which we shall now call $T^\varnothing$.
	Taking $T^\varnothing$ to be the least such gives $T^\varnothing$ a $\Sigma_2$ element of $\mathcal A$.

	The lemma tells us that $G(A;T^\varnothing)$ is not a win for $I$, so for any length-1 position $p_1\in T^*$, let $\tau(p_1)$ be some arbitrary, fixed move in $II$'s non-losing quasi-strategy.
	Now let $p_2$ be a length-2 move consistent with $\tau$ as defined so far and apply the Lemma with $B=A_1$ and $T=(T^\varnothing)_{p_2}$, yielding a quasi-strategy which we call $T^{p_2}$, $\Sigma_2$-definable from $(T^\varnothing)_{p_2}$ and hence an element of $\mathcal A$.
	Since the Lemma guarantees $I$ still has no winning strategy in $G(A;T^*_{p_2})$, for an arbitrary length-3 position $p_3$, let $\tau(p_3)$ be an arbitrary move consistent with $II$'s non-losing quasi-strategy there.

	Continuing in this way defines a strategy $\tau$ for $II$ by $\Sigma_2$-recursion, and as noted above, this gives $\tau\in\mathcal A$.
	All that remains is to show that $\tau$ is winning.
	If $x$ is consistent with $\tau$ then by construction it is in each $\plays{T^{p_{2n}}}$, and $\plays{T^{p_{2n}}}\cap B_n=\varnothing$ by the lemma, so $\forall n\in\omega(x\notin B_n)$, so $x\notin A$ and is a win for $II$.
\end{proof}
\let\bibliography\oldbibliography
\let\section\oldsection
\let\subsection\oldsubsection

Thus we have proved Theorem \ref{thm:main} part 3.

\subsection{\texorpdfstring{$\Det(n\hyp\wt\Pi^0_3)$}{Det n-tilde-Pi-0-3}}
\label{sub:n-Pi^0_3}

To prove this, we refer to the work of Montalb\'an and Shore, \cite{MontalbanShore2012}, Section 4, which in turn is a version of Martin's proof in Section 1.4 of \cite{MartinUP}.
Since that proof is long, needs very little modification and we have already seen how the arguments transfer from the ordinary case to determinacy of these auxiliary games, we will be brief.

Working in a minimal model $\mathcal{A}$ with $T^*\in \mathcal A$ and satisfying $\KP_{n+1}$, we have to ensure that all the objects defined in the proof of \cite{MontalbanShore2012} exist in $\mathcal A$.

First, let $A$ be as follows:
\begin{alignat*}{3}
	A &&& \text{ is $m\hyp\wt\Pi^0_3$ as witnessed by $\langle A_i\mid i\leq m\rangle$, a descending sequence of
		$\wt\Pi^0_3$ sets};\\
	A_i &=&& \bigcap_{j\in\omega} A_{i,j} \text{ for $\wt\Sigma^0_2$ sets $A_{i,j}$};\\
	A_{i,j} &=&& \bigcup_{k\in\omega} A_{i,j,k} \text{ for $\wt\Pi^0_1$ sets $A_{i,j,k}$}.
\end{alignat*}

Let $s$ be a position in $T^*$ of length at most $m$, and $S$ a subtree of $T^*$. Let $l=m-|s|$ and fix player $x$ to be $I$ if $l$ is even and $II$ otherwise, with $\bar x$ the opposite player.
Then let $B^x$ be $B$ if $x=I$ and $B^c$ otherwise, for any set $B$ (taking complementation inside the ambient space, which is $\plays {T^*}$ in the case of $A^x$, but when considering the winning set of a game in the tree $S$, it is in $\plays S$).
We define the predicates $P^s(S)$ by recursion up to $m$, exactly as in \cite{MontalbanShore2012}:
\begin{dfn}
	\label{dfn:P}~

	$P^\varnothing(S)$ holds iff there is a winning strategy for player $x$ in $G(A;S)$.

	$P^s(S)$ holds for $|s|=n+1$ iff there is a quasistrategy $U\subseteq S$ for player $x$ such that
	\begin{enumerate}
		\item $\plays U\subseteq A^x\cup A_{(m-n-1),s(n)}$; and
		\item $P^{s\upharpoonright n}(U)$ fails.
	\end{enumerate}
\end{dfn}

\begin{dfn}
	~

	A quasistrategy $U$ \emph{witnesses} that $P^s(S)$ if it is as required in Definition \ref{dfn:P}.

	A quasistrategy $U$ \emph{locally witnesses} that $P^s(S)$ if either $s=\varnothing$ and $U$ is a witness to $P^\varnothing(S)$, or $|s| > 0$, $U$ is a quasistrategy for player $x$ and there is a $D\subseteq S$ such that for every position $d\in D$, there is a quasistrategy $R^d\subseteq S_d$ for $\bar x$ such that:
	\begin{enumerate}
		\item $\forall d\in D\cap U(U_d\cap R^d \text{ witnesses that } P^s(R^d))$
		\item $\plays{U}\setminus\bigcup_{d\in D} \plays{R^d} \subseteq A^x$
		\item For a position $p\in S$ there is at most one position $d\in D$ such that $d\subseteq p \wedge p\in R^d$.
	\end{enumerate}
\end{dfn}

These are technical conditions that are used in the proof to ultimately ensure determinacy.
Many sets are defined using these conditions, so we will need to know the complexity, and thus that the sets are elements of $\mathcal A$.

\begin{prop}
	\label{prop:properties-complexities}
	~
	\begin{enumerate}
		\item $P^s(S)$ is $\Sigma_{|s|+1}$
		\item ``$U$ witnesses that $P^s(S)$ holds'' is $\Pi_{|s|}$
		\item ``$U$ locally witnesses that $P^s(S)$ holds'' is $\Sigma_{|s|+1}$.
	\end{enumerate}
\end{prop}
\begin{proof}
	Since ``there is a winning strategy for $x$ in $G(A;S)$'' is $\Sigma_1$ over $\mathcal A$ (recall Lemma \ref{lem:winning-is-delta-1}), $P^\varnothing(S)$ is $\Sigma_1$.
	Then if $|s|=n+1$, the definition requires checking that there is a $U$ such that $\neg P^{s\upharpoonright n}(U)$, so inductively the predicate overall is $\Sigma_{|s|+1}$.

	The complexities of the other predicates follow straightforwardly as well.
\end{proof}

Contrast this with Remark 4.3 of \cite{MontalbanShore2012}, where the same concepts are noted to be $\Sigma^1_{|s|+2}$, $\Pi^1_{|s|+1}$ and $\Sigma^1_{|s|+2}$, respectively; the Kleene-Basis theorem allows us to eliminate a quantifier in the base case.

With this observation, the remainder of Montalb\'an and Shore's proof can be followed in $\mathcal A$.
Note that, whenever we define a set (of plays, or of positions) it is, in their proof, at most $\Sigma^1_{n+2}$ using their Remark 4.3.
Consequently by our Proposition \ref{prop:properties-complexities} the same sets are at most $\Sigma^1_{n+1}$, indeed whenever their set is $\Sigma_{m+1}$, ours is $\Sigma_m$.
When it comes to define a function returning such sets, we use the fact that the $\mathcal A$-least set satisfying a $\Sigma_m$ property is again $\Sigma_m$.

Following their proof with these minor modifications, then, gives us:
\begin{thm}
	$\mathcal A\vDash\Det(n\hyp\wt\Pi^0_3)$.
\end{thm}

So we have proved Theorem \ref{thm:main}, part 4.

\subsection{\texorpdfstring{$\Det(\wt\Sigma^0_\alpha)$}{Det tilde-Sigma-0-alpha}}

For the final part of Theorem \ref{thm:main} we need to make a minor modification to the picture; Section \ref{sub:Effective Descriptive Set Theory in Uncountable Spaces} and Lemma \ref{lem:winning-set-complexity} only deal with sets of finite Borel rank.

\let\oldbibliography\bibliography
\def\bibliography#1{\relax}
\let\oldsection\section
\let\oldsubsection\subsection
\def\section#1{\let\section\subsubsection}
\let\subsection\subsubsection
\section{Setting}

We will first deal with the simpler case of part 6 of the main theorem, so let $\mathcal A$ be some $\mathcal A_\ZFC[\vec c]$, with $\vec c\supseteq\{\aleph_i\mid i\in\omega\}$.

\begin{dfn}
	Fix some $\Delta_1^{\mathcal H}$ encoding $p:\aleph_\omega\to T^*$ and let $A\subseteq \plays{T^*}$.

	\begin{enumerate}
		\item $A$ is $\wt\Sigma^0_1$ iff there is a partial $\Sigma_1(\mathcal H)$ function $e:\aleph_\omega\to\aleph_\omega$ where $A=\bigcup_\alpha\plays{T^*_{p(e(\alpha))}}$. The code for $A$ is then $\langle 0, e\rangle$, having coded $e$ as an element of $\aleph_\omega$.
			This is possible because there are only $\aleph_\omega$-many such functions that are $\Sigma_1(\mathcal H)$ definable with parameters from $\mathcal H$.
		\item $A$ is $\wt\Sigma^0_\alpha$ iff there is a partial $\Sigma_1(\mathcal H)$ function $e:\omega\to \aleph_\omega$ such that $e(a)$ is a code for a $\wt\Pi^0_{\beta}$ set for some $\beta<\alpha$ and $A=\bigcup_a e(a)$. The code for $A$ is then $\langle 1, e\rangle$.
		\item $A$ is $\wt\Pi^0_\alpha$ iff there is a $\wt\Sigma^0_\alpha$ set such that $A=B^c$. If $B$ has code $c$ then the code for $A$ is $\langle 2, c\rangle$.
	\end{enumerate}

\end{dfn}

\begin{thm}
	For $\alpha<\omega$, this definition of $\wt\Sigma^0_\alpha$ agrees with the old definition.
\end{thm}
\begin{proof}
	It is clear that the two definitions of $\wt\Sigma^0_1$ are equivalent:
	if we know the $\Sigma_1$ set $X$ of the original definition, then setting $e(\alpha)=\alpha \iff p(\alpha)\in X$ and leaving $e$ undefined otherwise, we have that $e$ is $\Sigma_1(\mathcal H)$ and the sets defined are the same.
	If we know $e$, and $e(\alpha)=\beta$, put $p(\beta)\in X$. 

	Now assume that we have shown the two definitions of $\wt\Sigma^0_n$ agree.
	Then if $A=\{x^*\mid\exists a R(x^*,a)\}$ for some $\wt\Pi^0_n$ predicate $R$, we take $e(a)$ to be (the code for) $\{x^*\mid R(x^*,a)\}$ and $A=\bigcup_a e(a)$.
	Noting that obtaining the codes for $R(\cdot, a)$ is $\Delta_1^\mathcal H$, we have that $e$ is a $\Delta_1^\mathcal H$ map, satisfying the requirements for the new definition.

	On the other hand if $A=\bigcup_a e(a)$ is $\wt\Sigma^0_{n+1}$ with the above definition we may assume that the relation $R(x^*,a)$ given by $x^*\in e(a)$ is $\wt\Pi^0_n$, and set $x^*\in A\iff \exists a(R(x^*,a))$.
\end{proof}

Thus if we fix $A^*\subseteq[T^*]$ to be $\wt\Sigma^0_\alpha$, we have $A^*\in\mathcal A$.
Furthermore, by the way it is defined, such sets are all within $\mathbf\Sigma^0_\alpha$.
Hence by the Borel Determinacy Theorem, $\mathcal A\vDash\Det(\wt\Sigma^0_\alpha)$.
By slightly altering Lemma \ref{lem:winning-set-complexity} we can see that, if $A\in\Sigma^0_\alpha$ then the corresponding winning set for the auxiliary game is $\wt\Sigma^0_\alpha$, and then the rest of the machinery of Theorem \ref{thm:ind-to-det} works as before.

For part 5 of the main theorem we just need to modify things slightly in light of Martin's level-by-level analysis of the determinacy of the Borel sets: 
\begin{thm}[\cite{MartinUP}, 2.3.5]
	\label{thm:borel-levels-det}
	Let $M$ be a model of $\mathsf{ZC}^- + \Sigma_1\hyp\mathsf{Replacement}$, and suppose $\mathcal{P}^\alpha(T)\in M$.
	Then, if $\alpha$ is finite, $M\vDash$``all $\mathbf\Delta^0_{\alpha+4}$ games in $T$ are determined''.
	If $\alpha$ is an infinite countable ordinal of $M$ then $M\vDash$``all $\mathbf\Delta^0_{\alpha+3}$ games are determined''.
\end{thm}
(Note that here we mean $\mathbf\Delta^0_\alpha$ in the sense of $M$, so while this may not be all of the true $\mathbf\Delta^0_\alpha$ sets we know that in our situation, with $M$ transitive, it will be all of the $\wt\Delta^0_\alpha$ sets.)

Thus we will want $\mathcal A$ to be a model of $\mathsf{ZC}^-+\Sigma_1\hyp\mathsf{Replacement}+\mathcal{P}^\alpha(T^*)$ exists (where $T^*$ is the auxiliary game tree, $(\omega\times\aleph_\omega)^{<\omega}$).

Now consider the argument again, starting with a model $M$ with measurable cardinal $\kappa$, of $\ZFC^-+\mathcal{P}^\alpha(\kappa)$ exists.
We iterate this model as in Section \ref{sec:Obtaining Indiscernibles} to obtain $M_\lambda$, which is elementarily equivalent to $M$, and then force to produce models $\mathcal A[\vec c]$.
By elementarity, $M_\lambda\vDash\ZFC^-+\mathcal{P}^\alpha(\kappa_\lambda)$ exists, and so by the usual argument for showing that the Power Set axiom holds in generic extensions, so does $M_\lambda[\vec c]$.
Thus all the models $\mathcal A[\vec c]=M_\lambda[\vec c]$ which we use in the arguments of Section \ref{sec:Determinacy from Indiscernibility} see that $\mathcal P^\alpha(T^*)$ exists, since in that case $\kappa_\lambda=\aleph_\omega$, hence the hypotheses of Theorem \ref{thm:borel-levels-det} are satisfied when it comes to proving determinacy of the auxiliary game.
\let\bibliography\oldbibliography
\let\section\oldsection
\let\subsection\oldsubsection

We have thus completed the proof Theorem \ref{thm:main}.

\section{Conclusions and Open Questions}
\label{sec:Conclusions and Open Questions}

The main result shows the adaptability of the technique originally invented by Martin for proving determinacy results on $\alpha\hyp\mathbf\Pi^1_1$.
The basic results of Section \ref{sub:Dual Classes in the Difference Hierarchy} show that we automatically get the determinacy of $\os+\Pi^0_2$ from $\skp$ (and so on) as we would expect.
In this way we have established new upper-bounds on the consistency strength of determinacy of classes strictly between $\os$ and $(\omega^2+1)\hyp\Pi^1_1$.

Open questions are of two kinds; the first is whether the hypotheses used here are minimal and the second is what hypotheses can we find that prove the determinacy of similar pointclasses.
To the second question, we have essentially exhausted the possibilities for extending \os{} by one more set, but we could investigate the modifications to the difference hierarchy in \cite{DuBose90} to see if more results are provable.

To the first question, we should not expect optimality in parts 1--3 of \ref{thm:main}, since for those we either know that the determinacy proof holds in models weaker than the forcing extensions we consider, or have no reason to suspect otherwise.
The problem is that we cannot preserve arbitrary theories when iterating and doing the forcing in Section \ref{sec:Obtaining Indiscernibles}, so even where we have optimal determinacy results (for instance, it is known that $\Det(\Sigma^0_2)$ is equivalent to closure under $\Sigma^1_1$-monotone inductive definitions) we cannot necessarily transfer them.

In the case of part 4, Montalb\'an and Shore show in \cite{MontalbanShore2012} that we cannot even find an exact characterisation of Determinacy down in the Borel hierarchy, so we should not hope to find any in the refined difference hierarchy.
However, we may rather hope to prove, as they did, that no exact correspondence can be found.

On the other hand we know that Borel determinacy requires almost the full strength of \ZFC, with Martin's analysis showing that we should expect parts 5 and 6 to be nearly optimal.

\bibliography{bibliography}
\end{document}